\documentclass[final,5p,times]{elsarticle}
\usepackage{amsmath,amssymb,amsfonts}
\usepackage{amsthm}
\usepackage{mathtools}

\usepackage{graphicx}
\graphicspath{{Figures/}}

\usepackage{booktabs}
\usepackage{multirow}
\usepackage{float}
\usepackage{subfig}

\usepackage{algorithm}
\usepackage{algpseudocode}

\usepackage{tikz}
\usetikzlibrary{
calc,
arrows,
automata,
backgrounds,
petri,
decorations.pathmorphing,
shapes,
shapes.geometric,
positioning
}

\usepackage[nameinlink]{cleveref}
\usepackage{placeins}

\newtheorem{theorem}{Theorem}[section]
\newtheorem{proposition}[theorem]{Proposition}
\newtheorem{lemma}[theorem]{Lemma}
\newtheorem{corollary}[theorem]{Corollary}
\newtheorem{definition}[theorem]{Definition}
\newtheorem{remark}[theorem]{Remark}
\newtheorem{assumption}[theorem]{Assumption}

\journal{Aerospace Science and Technology}

\begin{document}

%%%%%%%%%%%%%%%%%%%%%%%%%%%%%%%%%%%%%%%%%%%%%%%%%%%%%%%%%%%%%%%%%%%%%%%%%%%%%%
%% Front Matter
%%%%%%%%%%%%%%%%%%%%%%%%%%%%%%%%%%%%%%%%%%%%%%%%%%%%%%%%%%%%%%%%%%%%%%%%%%%%%%

\begin{frontmatter}

%%%%%%%%%%%%%%%%%%%%%%%%%%%%%%%%%%%%%%%%%%%%%%%%%%%%%%%%%%%%%%%%%%%%%%%%%%%%%%
%% Title
%%%%%%%%%%%%%%%%%%%%%%%%%%%%%%%%%%%%%%%%%%%%%%%%%%%%%%%%%%%%%%%%%%%%%%%%%%%%%%

\title{
Transient Stability of Nonlinear Stochastic Dynamics:
A Finite-Amplitude Logarithmic Measure Framework
}

%%%%%%%%%%%%%%%%%%%%%%%%%%%%%%%%%%%%%%%%%%%%%%%%%%%%%%%%%%%%%%%%%%%%%%%%%%%%%%
%% Authors
%%%%%%%%%%%%%%%%%%%%%%%%%%%%%%%%%%%%%%%%%%%%%%%%%%%%%%%%%%%%%%%%%%%%%%%%%%%%%%

\author[ursc]{Surya Ratna Prakash D}
\ead{dsrp@ursc.gov.in}

\author[iisc]{Soumyendu Raha\corref{cor1}}
\ead{raha@iisc.ac.in}

\cortext[cor1]{Corresponding author.}

%%%%%%%%%%%%%%%%%%%%%%%%%%%%%%%%%%%%%%%%%%%%%%%%%%%%%%%%%%%%%%%%%%%%%%%%%%%%%%
%% Affiliations
%%%%%%%%%%%%%%%%%%%%%%%%%%%%%%%%%%%%%%%%%%%%%%%%%%%%%%%%%%%%%%%%%%%%%%%%%%%%%%

\affiliation[ursc]{
organization={U R Rao Satellite Centre, Indian Space Research Organisation (ISRO)},
city={Bengaluru},
postcode={560017},
country={India}
}

\affiliation[iisc]{
organization={Centre for Computational and Data Sciences, Indian Institute of Science},
city={Bengaluru},
postcode={560012},
country={India}
}

%%%%%%%%%%%%%%%%%%%%%%%%%%%%%%%%%%%%%%%%%%%%%%%%%%%%%%%%%%%%%%%%%%%%%%%%%%%%%%
%% Abstract
%%%%%%%%%%%%%%%%%%%%%%%%%%%%%%%%%%%%%%%%%%%%%%%%%%%%%%%%%%%%%%%%%%%%%%%%%%%%%%

\begin{abstract}

Spacecraft guidance, navigation, and control systems operating during mission-critical phases such as atmospheric entry, powered descent, and planetary landing require reliable assessment of finite-time transient behavior under stochastic uncertainty. Existing stochastic stability frameworks primarily characterize asymptotic or infinitesimal behavior and therefore provide limited insight into finite-amplitude transient amplification over operationally relevant time horizons.

This paper develops a finite-amplitude logarithmic measure for nonlinear It\^o stochastic differential equations, extending classical matrix measures from infinitesimal to finite-amplitude perturbation evolution. The proposed framework establishes finite-time mean and variance bounds for logarithmic perturbation growth, derives Chernoff-type probabilistic bounds on transient amplification, and shows that mean transient stability does not necessarily guarantee finite-time pathwise safety. The theory is further extended to projected stochastic dynamics, leading to a transient-risk index that quantitatively balances deterministic contraction and diffusion-induced variability for transient-risk-aware system design.

The proposed framework is validated through Monte Carlo simulations and flight-like lunar descent telemetry. The results demonstrate substantially improved prediction of nonlinear transient growth compared with classical Jacobian-based approaches and successfully distinguish trajectories exhibiting similar nominal behavior but significantly different transient-risk characteristics. The proposed framework provides a unified methodology for finite-time transient stability analysis, probabilistic transient-risk assessment, and transient-risk-aware guidance, navigation, and control of nonlinear stochastic aerospace systems.

\end{abstract}

%%%%%%%%%%%%%%%%%%%%%%%%%%%%%%%%%%%%%%%%%%%%%%%%%%%%%%%%%%%%%%%%%%%%%%%%%%%%%%
%% Highlights
%%%%%%%%%%%%%%%%%%%%%%%%%%%%%%%%%%%%%%%%%%%%%%%%%%%%%%%%%%%%%%%%%%%%%%%%%%%%%%

\begin{highlights}

\item Introduces a finite-amplitude logarithmic measure for nonlinear It\^o stochastic systems.

\item Develops deterministic and probabilistic finite-time transient stability theory.

\item Extends transient stability analysis to projected stochastic dynamics.

\item Proposes a transient-risk index for transient-risk-aware system design.

\item Validates the framework using Monte Carlo simulations and flight-like lunar descent telemetry.

\end{highlights}

%%%%%%%%%%%%%%%%%%%%%%%%%%%%%%%%%%%%%%%%%%%%%%%%%%%%%%%%%%%%%%%%%%%%%%%%%%%%%%
%% Keywords
%%%%%%%%%%%%%%%%%%%%%%%%%%%%%%%%%%%%%%%%%%%%%%%%%%%%%%%%%%%%%%%%%%%%%%%%%%%%%%


\end{frontmatter}

%%%%%%%%%%%%%%%%%%%%%%%%%%%%%%%%%%%%%%%%%%%%%%%%%%%%%%%%%%%%%%%%%%%%%%%%%%%%%%
%% Main Paper
%%%%%%%%%%%%%%%%%%%%%%%%%%%%%%%%%%%%%%%%%%%%%%%%%%%%%%%%%%%%%%%%%%%%%%%%%%%%%%

%=====================================================
%=====================================================	
\section{Introduction}
\label{sec:introduction}

Spacecraft guidance, navigation, and control (GNC) systems operate under significant uncertainty during mission-critical operations such as atmospheric entry, powered descent, planetary landing, rendezvous, and autonomous docking. During these phases, small state perturbations may undergo
substantial transient amplification over short time horizons, producing safety-critical deviations before asymptotic stability mechanisms become dominant. Quantifying finite-time stochastic perturbation growth is therefore fundamental for the design of robust autonomous guidance, navigation, and control algorithms for future space missions.

Classical stochastic stability theory has primarily focused on asymptotic, exponential, or moment-based notions of stability, typically established through Lyapunov methods, martingale techniques, or second-moment estimates
\cite{Mao2007,Arnold1998,Khasminskii2012}. Although mathematically well established these approaches primarily characterize asymptotic behavior and provide limited insight into finite-time transient perturbation growth.

Finite-time and fixed-time stabilization have been
extensively studied for deterministic and stochastic
systems \cite{Lee2022,Polyakov2012,Zheng2022,zuo2022fixed}.
However, these frameworks address convergence
properties rather than finite-time perturbation
amplification in nonlinear continuous-time It\^o
systems.

A complementary perspective is provided by
incremental stability and contraction theory
\cite{Angeli2002,LohmillerSlotine1998,
ForniSepulchre2014Auto,Pavlov2008,
ManchesterSlotine2017,DiBernardo2014,
Kawano2024,Kawano2025,
aminzare2022,li2024stochastic},
which characterize trajectory-wise perturbation
dynamics. Related studies on non-normal systems
\cite{Trefethen1993,Schmid2007}
demonstrate that substantial transient amplification
may occur despite asymptotic stability. However,
these approaches remain fundamentally based on
infinitesimal perturbations and do not characterize
finite-amplitude nonlinear transient growth.

Existing stochastic logarithmic norm formulations
\cite{ahmad2010estimation,aminzare2022}
are effective for linear and weakly nonlinear
systems. Extending them to computationally tractable
finite-time transient analysis of strongly nonlinear
It\^o stochastic systems without Jacobian
linearization remains an open problem.

These limitations are particularly important in
projection-based navigation and estimation systems,
where continuous data assimilation modifies both the
system dynamics and transient stability properties
\cite{Blackmore2010,Sikorski2021,Teel2014,
CrassidisJunkins2012,ReichCotter2015}. Similar
challenges arise in atmospheric entry,
powered-descent guidance, and planetary landing
systems operating under significant uncertainty,
where robust guidance and trajectory optimization
are essential for reliable mission performance
\cite{Lu2018,Sagliano2021,Xiong2022}.

To address these challenges, this paper develops a
unified finite-amplitude transient stability
framework for nonlinear It\^o stochastic systems,
establishing deterministic and probabilistic
finite-time transient bounds, extending the analysis
to projected dynamics, and introducing a
transient-risk index for risk-aware design.

The principal contributions are summarized as follows.
\begin{itemize}
\item \textbf{Finite-amplitude logarithmic measure.}
A logarithmic measure for nonlinear It\^o stochastic differential equations extending classical matrix measures from infinitesimal to finite-amplitude perturbation evolution.

\item \textbf{Unified finite-time transient stability theory.}
Finite-time mean, variance, and probabilistic transient-growth bounds together with a transient-risk index unifying deterministic contraction and diffusion-induced variability.

\item \textbf{Projected stochastic dynamics.}
Extension to projected stochastic dynamics with transient-risk-aware design, validated through Monte Carlo simulations and flight-like lunar descent telemetry.
\end{itemize}

Figure~\ref{fig:framework} summarizes the proposed
computational pipeline.

\begin{figure}[t]
\centering
\includegraphics[width=0.8\linewidth]{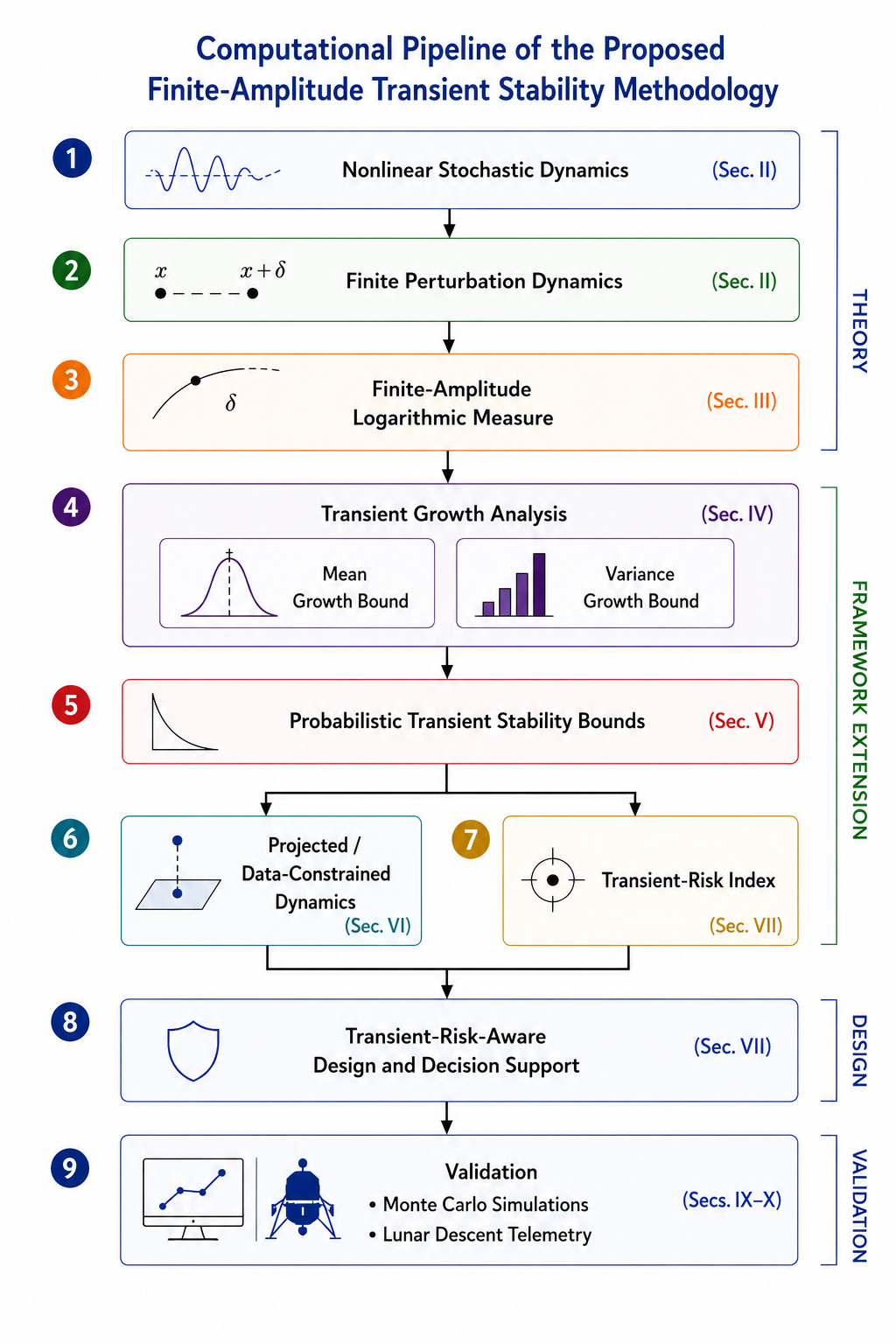}
\caption{Computational pipeline of the proposed finite-amplitude transient stability framework. Section numbers indicate where each stage of the methodology is developed.}
\label{fig:framework}
\end{figure}
%=====================================================
%=====================================================	

%=====================================================
%=====================================================	
\section{Preliminaries and Problem Statement}
\label{sec:preliminaries}

This section introduces the stochastic dynamical model,
finite-perturbation formulation, and transient-stability
problem underlying the proposed framework.

%-----------------------------------------------------
\subsection{Notation and Probability Space}

Let $(\Omega,\mathcal{F},\mathbb{P})$ be a complete probability
space equipped with a right-continuous filtration
$\{\mathcal{F}_t\}_{t\ge0}$. Expectations and variances are
denoted by $\mathbb{E}[\cdot]$ and $\operatorname{Var}(\cdot)$, respectively. For $x\in\mathbb{R}^n$, $\|x\|_p$ denotes the $\ell_p$ norm, with $\|\cdot\|$ denoting the Euclidean norm when no ambiguity arises. The identity matrix is denoted by $I$.

%-----------------------------------------------------
\subsection{Stochastic Differential Equation}

We consider nonlinear It\^o stochastic differential equations of the form
\begin{equation}
dX_t = f(X_t,t)\, dt + \sum_{j=1}^m
b_j(X_t,t)\, dW_t^j,
\label{eq:sde_general}
\end{equation}
where $X_t\in\mathbb{R}^n$ is the system state,
$f:\mathbb{R}^n\times\mathbb{R}_+\rightarrow\mathbb{R}^n$
is the drift field, $b_j:\mathbb{R}^n\times\mathbb{R}_+\rightarrow\mathbb{R}^n$
are diffusion fields, and $\{W_t^j\}_{j=1}^m$ are independent Wiener processes. Throughout, stochastic integrals are interpreted in
the It\^o sense, and all processes are assumed adapted
to $\{\mathcal F_t\}$ \cite{Mao2007}.

\begin{assumption}
\label{assum_reg}

The drift and diffusion mappings
$f(\cdot,t)$ and $b_j(\cdot,t)$,
$j=1,\ldots,m$, are locally Lipschitz in $x$
and satisfy the linear growth condition

\begin{equation}
\|f(x,t)\|
+
\sum_{j=1}^{m}
\|b_j(x,t)\|
\le
K(1+\|x\|)
\end{equation}

for some constant $K>0$.
\end{assumption}

Under Assumption~\ref{assum_reg}, \eqref{eq:sde_general}
admits a unique strong solution on every finite time interval
almost surely \cite{Mao2007,Arnold1998}. Consequently,
both the nominal trajectory and finite perturbation
trajectories are well defined throughout the analysis.
%-----------------------------------------------------
\subsection{Perturbation Dynamics}
To characterize finite-time transient behavior, we consider finite perturbations evolving under the same realization of the driving Wiener process. Let $\delta X_t$ denote a non-zero finite perturbation of the state trajectory $X_t$ along a given sample path. The corresponding perturbation dynamics are defined as
\begin{equation}
d(\delta X_t)
=
\delta f_t\, dt
+
\sum_{j=1}^m
\delta b_{j,t}\, dW_t^j,
\label{eq:perturbation_general}
\end{equation}
where
\begin{align}
\delta f_t &:= f(X_t+\delta X_t,t) - f(X_t,t), \notag\\
\delta b_{j,t} &:= b_j(X_t+\delta X_t,t) - b_j(X_t,t).
\end{align}

Equation~\eqref{eq:perturbation_general}
defines the exact nonlinear perturbation dynamics without
Jacobian linearization or variational equations.

Since perturbation amplification is inherently multiplicative,
transient growth is characterized through the logarithm of the
perturbation norm.
\begin{definition}[Transient Stability]
Let $\delta X_t$ denote the perturbation process associated with
\eqref{eq:sde_general}. The system is said to be

\begin{itemize}

\item \emph{mean transiently stable on $[0,T]$} if

\[
\mathbb{E}\!\left[ \ln \|\delta X_{t+\Delta t}\|_p -
\ln \|\delta X_t\|_p \right] \le 0,
\]
for all sufficiently small $\Delta t > 0$ and all $t \in [0,T]$;
\item \emph{pathwise transiently stable with confidence level
$1-\epsilon$ on $[0,T]$} if
\[
\mathbb{P}\!\left( \sup_{t\in[0,T]}
\ln \|\delta X_t\|_p
\le \ln \|\delta X_0\|_p \right)
\ge 1-\epsilon .
\]
\end{itemize}
The objective of this paper is to characterize finite-time
perturbation growth on nonlinear stochastic flows through a
finite-amplitude logarithmic measure. To this end, we derive
deterministic and probabilistic finite-time bounds for
$\ln\|\delta X_t\|_p$, which form the basis for the subsequent
theoretical developments.
\end{definition}

%=====================================================
%=====================================================

%=====================================================
%=====================================================
\section{Logarithmic Norm in the Lipschitz Sense}
\label{sec:lognorm}
This section introduces a finite-amplitude logarithmic
measure for nonlinear mappings that extends classical
matrix measures from infinitesimal to finite perturbation
evolution.
%-----------------------------------------------------
\subsection{Definition}
Let $\phi:U\subset\mathbb{R}^n\rightarrow\mathbb{R}^n$
be locally Lipschitz. For any finite perturbation
$\delta x\neq0$, define \( \delta\phi := \phi(x+\delta x)-\phi(x).\)

\begin{definition}[Logarithmic Norm Induced by the $\ell_p$-Norm]
\label{def:lognorm}
\begin{align}
\mu_p(\phi)
:=
\limsup_{\substack{x,x+\delta x \in U \\ \|\delta x\|_p>0 \\ h \downarrow 0}}
\frac{\|\delta x + h\,\delta \phi\|_p - \|\delta x\|_p}
{h\,\|\delta x\|_p}.
\label{eq:lognorm_def}
\end{align}
\end{definition}

Unlike the classical matrix measure, Definition~\ref{def:lognorm}
is formulated directly on finite perturbation increments.

\begin{theorem}[Well-Posedness]
\label{thm:wellposed_mu}
Let $\phi:U\rightarrow\mathbb{R}^n$
be locally Lipschitz on an open set
$U\subset\mathbb{R}^n$. Then the logarithmic measure
$\mu_p(\phi)$ defined by \eqref{eq:lognorm_def}
is well defined and finite.
\end{theorem}

\begin{proof}
Since $\phi$ is locally Lipschitz,
\[
\|\delta\phi\|_p\le L_p(\phi)\|\delta x\|_p.
\]
The reverse triangle inequality gives
\[
\left|
\frac{\|\delta x+h\delta\phi\|_p-\|\delta x\|_p}
{h\|\delta x\|_p}
\right|
\le L_p(\phi),
\]
so the upper Dini derivative in
Definition~\ref{def:lognorm} is finite.
\end{proof}

An equivalent representation is
\begin{equation}
\mu_p(\phi)
=
\limsup_{h \to 0}
\frac{L_p(I + h\phi) - 1}{h},
\label{eq:lognorm_lipschitz}
\end{equation}
whenever the induced Lipschitz constant
$L_p(\cdot)$ is well defined. This is the nonlinear analogue of the classical matrix-measure formula.

\begin{theorem}[Consistency with Classical Matrix Measures]
\label{thm:consistency}
Suppose \( \phi(x)=Ax \) for some matrix $A\in\mathbb{R}^{n\times n}$. Then the nonlinear logarithmic measure defined by
\eqref{eq:lognorm_def}
coincides with the classical logarithmic norm,
\[
\mu_p(\phi) = \mu_p(A).
\]

\end{theorem}

\begin{proof}
For $\phi(x)=Ax$, we have
$\delta\phi=A\delta x$.
Substituting into
\eqref{eq:lognorm_def}
gives
\[
\mu_p(\phi) =
\limsup_{h\to0}
\frac{\|I+hA\|_p-1}{h},
\]
which is precisely the Dahlquist matrix measure.
\end{proof}

\begin{theorem}[Local Jacobian Representation]
\label{thm:jacobian_representation}

Assume that
$\phi$
is continuously differentiable. Then
\[
\delta\phi = J_{\phi}(x)\,\delta x
+ o\!\left(\|\delta x\|\right),
\qquad
\text{as } \|\delta x\|\to 0.
\]

Consequently,
\[
\mu_p(\phi)
\le
\sup_{x\in U}
\mu_p\!\left(
J_\phi(x)
\right)
\] in the infinitesimal perturbation limit.

\end{theorem}

\begin{proof}
The first-order Taylor expansion gives
\[
\phi(x+\delta x) =
\phi(x)+J_\phi(x)\delta x+o(\|\delta x\|).
\]
Substituting into
Definition~\ref{def:lognorm}
and letting
$\|\delta x\|\to0$
recovers the classical Jacobian logarithmic measure.
\end{proof}

\begin{proposition}[Finite-Amplitude Deviation]
\label{prop:finite_amplitude_deviation}

Assume that
$\phi\in C^{2}(U)$. Then, for finite perturbations,

\[ \mu_p(\phi) = \mu_p\!\left(J_\phi(x)\right)
+ O(\|\delta x\|). \]
\end{proposition}

\begin{proof}
The second-order Taylor expansion gives
\[ \delta\phi = J_\phi(x)\delta x + O(\|\delta x\|^2). \]
Substituting into
Definition~\ref{def:lognorm}
yields
\[ \mu_p(\phi) = \mu_p(J_\phi(x)) + O(\|\delta x\|), \]
which proves the result.
\end{proof}
%-----------------------------------------------------
\subsection{First-Order Characterization}
A first-order expansion in the upper Dini derivative sense
admits the representation
\begin{equation}
\mu_p(\phi)
=
\limsup_{\substack{x,x+\delta x \in U \\ \|\delta x\|_p>0}}
\frac{1}{\|\delta x\|_p^p}
\sum_{i=1}^n
\operatorname{sign}(\delta x_i)\,
|\delta x_i|^{p-1}\,
\delta \phi_i,
\label{eq:lognorm_first}
\end{equation}
or, equivalently,
\begin{equation}
\mu_p(\phi)
=
(\delta \phi)^{\top}
\frac{\partial}{\partial(\delta x)}
\ln\!\big(\|\delta x\|_p\big),
\label{eq:lognorm_gradient}
\end{equation}
where
\begin{equation}
\frac{\partial}{\partial(\delta x)}
\ln\!\big(\|\delta x\|_p\big)
=
\frac{1}{\|\delta x\|_p^p}
\begin{bmatrix}
\delta x_1 |\delta x_1|^{p-2} \\
\vdots \\
\delta x_n |\delta x_n|^{p-2}
\end{bmatrix}.
\label{eq:lognorm_grad_explicit}
\end{equation}
We denote this gradient by
\(
g := \frac{\partial}{\partial(\delta x)} \ln \|\delta x\|_p.
\)
%-----------------------------------------------------
%-----------------------------------------------------
\subsection{Gradient and Hessian Properties}

Let \( V(\delta x) = \ln\|\delta x\|_p, \qquad \delta x \neq 0. \)

The gradient of $V$ is given by

\begin{equation}
\nabla V(\delta x)
=
\frac{1}{\|\delta x\|_p^{p}}
\begin{bmatrix}
\delta x_1 |\delta x_1|^{p-2}\\
\vdots\\
\delta x_n |\delta x_n|^{p-2}
\end{bmatrix},
\label{eq:gradient_log_norm}
\end{equation}

The Hessian exists almost everywhere for
$\delta x \neq 0$ and is given componentwise by

\begin{equation}
\begin{aligned}
\big[\nabla^2 V(\delta x)\big]_{ij}
&=
(p-1)
\frac{
|\delta x_i|^{p-2}
}
{
\|\delta x\|_p^{p}
}
\delta_{ij}
\\
&\quad
-
p
\frac{
\delta x_i |\delta x_i|^{p-2}
\,
\delta x_j |\delta x_j|^{p-2}
}
{
\|\delta x\|_p^{2p}
}.
\end{aligned}
\label{eq:hessian_components}
\end{equation}

where
$\delta_{ij}$
denotes the Kronecker delta.

Equivalently,

\begin{equation}
\nabla^2 V(\delta x)
=
(p-1)
\frac{
\operatorname{diag}
\!\left(
|\delta x|^{p-2}
\right)
}
{
\|\delta x\|_p^{p}
}
-
p\,gg^\top,
\label{eq:hessian_matrix}
\end{equation}

where
$g$
is defined in
\eqref{eq:gradient_log_norm}.

\begin{lemma}[Hessian Bound]
\label{lem:hessian_bound}

There exists a constant
$C_p>0$
depending only on the norm parameter
$p$
such that

\begin{equation}
\|
\nabla^2 V(\delta x)
\|
\le
\frac{C_p}
{\|\delta x\|^{2}}
,
\qquad
\delta x\neq0.
\label{eq:hessian_bound}
\end{equation}

\end{lemma}

\begin{proof}
Equation~\eqref{eq:hessian_matrix} shows that every Hessian term scales as
$\|\delta x\|^{-2}$.
Norm equivalence in finite-dimensional spaces then yields
\[
\|\nabla^2V(\delta x)\|
\le
\frac{C_p}{\|\delta x\|^2},
\]
for some constant $C_p>0$.
\end{proof}

%-----------------------------------------------------
\subsection{Structural Properties}

\begin{proposition}[Subadditivity]
For any locally Lipschitz mappings \(\phi_1\) and \(\phi_2\),
the logarithmic norm satisfies the conservative bound
\begin{equation}
\mu_p(\phi_1+\phi_2)
\le
\mu_p(\phi_1)
+
L_p(\phi_2),
\label{eq:lognorm_subadd}
\end{equation}
where \(L_p(\phi_2)\) denotes the \(\ell_p\)-Lipschitz constant of
\(\phi_2\).
\end{proposition}

\begin{proof}

From Definition~\eqref{eq:lognorm_def},
\[
\mu_p(\phi_1+\phi_2)
=
\limsup_{h\to0^+}
\frac{
\|\delta x + h(\delta\phi_1+\delta\phi_2)\|_p
-
\|\delta x\|_p
}{
h\|\delta x\|_p
}.
\]

Using the triangle inequality,
\[
\|\delta x + h(\delta\phi_1+\delta\phi_2)\|_p
\le
\|\delta x + h\delta\phi_1\|_p
+
h\|\delta\phi_2\|_p.
\]

Hence,
\(
\mu_p(\phi_1+\phi_2)
\le
\mu_p(\phi_1)
+
\limsup
\frac{
\|\delta\phi_2\|_p
}{
\|\delta x\|_p
}.
\)

Using Lipschitz continuity of \(\phi_2\),
\(
\|\delta\phi_2\|_p
\le
L_p(\phi_2)\|\delta x\|_p,
\)
which yields the result.
\end{proof}

%--------------------------------------------

\begin{proposition}[Positive Homogeneity]
\label{prop:homogeneity}

For every scalar $c\ge0$,

\[
\mu_p(c\phi) = c\,\mu_p(\phi).
\]

\end{proposition}
\begin{proof}
For $c=0$ the result is immediate.
For $c>0$, letting $\tilde h=ch$ in
Definition~\ref{def:lognorm}
gives
\[
\mu_p(c\phi)=c\mu_p(\phi).
\]
\end{proof}

%--------------------------------------------

\begin{proposition}[Logarithmic Norm Upper Bound]
If the nonlinear mapping
\(\phi:\mathbb{R}^n\to\mathbb{R}^n\)
is Lipschitz continuous with respect to the
\(\ell_p\)-norm with Lipschitz constant
\(L_p(\phi)\), then
\begin{equation}
\mu_p(\phi)
\le
L_p(\phi).
\label{eq:lognorm_bounds}
\end{equation}
\end{proposition}

\begin{proof}
By Lipschitz continuity,
\(
\|\delta\phi\|_p
\le
L_p(\phi)\|\delta x\|_p.
\)

Using Definition~\eqref{eq:lognorm_def} together with the
triangle inequality,
\[
\|\delta x+h\delta\phi\|_p
\le
(1+hL_p(\phi))\|\delta x\|_p.
\]

Dividing by \(h\|\delta x\|_p\) and taking
\(h\to0^+\) yields
\[
\mu_p(\phi)\le L_p(\phi).
\]
\end{proof}
%=====================================================
%=====================================================	

%=====================================================
%=====================================================	
\section{Transient Growth of Stochastic Perturbations}
\label{sec:transient}
This section develops the finite-time transient stability
theory underlying the proposed framework. Building upon
the finite-amplitude logarithmic measure of
Section~\ref{sec:lognorm}, we derive explicit mean and
variance bounds for logarithmic perturbation growth,
forming the basis for the probabilistic analysis in
Section~\ref{sec:probabilistic}.

\subsection{Perturbation Dynamics and Logarithmic Growth}
Applying It\^o's formula to the logarithmic perturbation norm yields
\begin{align}
	 d\big(\ln\|\delta X_t\|_p\big) =&
	\Big(
	\delta f_t^{\top}
	\nabla_{\delta x}\ln\|\delta x\|_p
	+ \nonumber\\ &\quad \frac{1}{2}
	\sum_{j=1}^m
	\delta b_{j,t}^{\top}
	\nabla^2 _{\delta x}\ln\|\delta x\|_p
	\,\delta b_{j,t}
	\Big) dt
    \nonumber\\ &\quad
	+ \sum_{j=1}^m
	\delta b_{j,t}^{\top}
	\nabla_{\delta x}\ln\|\delta x\|_p
\, dW_t^j,
	\label{eq:ito_log_norm}
\end{align}

%-----------------------------------------------------
\subsection{Mean Transient Growth Bound}
Taking expectations in \eqref{eq:ito_log_norm}
eliminates the martingale term and yields the following
upper bound on the expected logarithmic perturbation
growth rate.

\begin{theorem}[Mean Transient Growth Bound]
\label{thm:mean_bound}

Assume that each diffusion field
$b_j$ is locally Lipschitz with constant $L_{b_j}$. Then
\begin{equation}
\mathbb E
\!\left[
d\ln\|\delta X_t\|_p
\right]
\le
\left(
\mu_p(f)
+
\frac{C_p}{2}
\sum_{j=1}^{m}
L_{b_j}^{\,2}
\right)
dt ,
\label{eq:mean_bound}
\end{equation}

where
\[
\alpha :=
\mu_p(f)
+
\frac{C_p}{2}
\sum_{j=1}^{m}
L_{b_j}^{\,2},
\]

and
\(C_p\)
is the constant introduced in
Lemma~\ref{lem:hessian_bound}.

\end{theorem}

\begin{proof}
Applying It\^o's formula to
$\ln\|\delta X_t\|_p$
and taking expectations eliminates the martingale
term, yielding
\[
\mathbb{E}\!\left[
d\ln\|\delta X_t\|_p
\right] = \left( \delta f_t^\top g
+ \frac12 \sum_{j=1}^{m} \delta b_{j,t}^\top
\nabla^2V\, \delta b_{j,t}
\right)dt .
\]

From the first-order characterization,
\[
\delta f_t^\top g \le \mu_p(f),
\]
while Lemma~\ref{lem:hessian_bound} together with
\[
\|\delta b_{j,t}\| \le L_{b_j}\|\delta x\|
\]
implies
\[
\delta b_{j,t}^{\top} \nabla^2V\, \delta b_{j,t}
\le C_pL_{b_j}^{\,2}.
\]
Substituting these bounds into \eqref{eq:ito_log_norm} gives \eqref{eq:mean_bound}.
\end{proof}

\begin{corollary}[Sufficient Mean Stability Condition]
\label{cor:mean_nonpositive}

If
\begin{equation}
\mu_p(f) + \frac{C_p}{2} \sum_{j=1}^{m} L_{b_j}^{\,2} \le 0,
\label{eq:sufficient_condition}
\end{equation}

then \[\mathbb{E} \!\left[ d\ln\|\delta X_t\|_p \right] \le 0. \]
Consequently, the expected logarithmic perturbation
norm is non-increasing over finite time intervals.
\end{corollary}
\begin{proof}
The result follows immediately from
Theorem~\ref{thm:mean_bound} by integrating
\eqref{eq:sufficient_condition}
over any finite interval.
\end{proof}

% \begin{proof}
% The result follows immediately from
% Theorem~\ref{thm:mean_bound}.
% If \eqref{eq:sufficient_condition} holds, then
% \[
% \mathbb{E}\!\left[ d\ln\|\delta X_t\|_p
% \right] \le0.
% \]
% Integrating over any finite interval $[0,T]$ gives
% \[
% \mathbb{E}\!\left[
% \ln\|\delta X_T\|_p
% \right] \le
% \mathbb{E}\!\left[ \ln\|\delta X_0\|_p \right],
% \]
% which completes the proof.
% \end{proof}
%----------------------------------------------------
\subsection{Variance Bound}
The following proposition bounds the diffusion-induced
variability of logarithmic perturbation growth.

\begin{proposition}[Variance Bound]
\label{prop:variance}

The infinitesimal variance growth of the logarithmic
perturbation norm satisfies

\begin{equation}
\operatorname{Var}
\!\left( d\ln\|\delta X_t\|_p
\right) \le K_p^2 \sum_{j=1}^{m}
L_{b_j}^{\,2} \,dt ,
\label{eq:variance_bound}
\end{equation}

where \(K_p>0\) depends only on the chosen $\ell_p$ norm.

\end{proposition}

\begin{proof}
From \eqref{eq:ito_log_norm}, the stochastic contribution is
\[
\sum_{j=1}^{m} g^\top\delta b_{j,t}\,dW_t^j,
\]
where
\[
g=\nabla_{\delta x}\ln\|\delta x\|_p.
\]
Using the quadratic variation of It\^o integrals,
\[
\operatorname{Var}\!\left(
d\ln\|\delta X_t\|_p
\right) = \sum_{j=1}^{m} \mathbb E\!\left[
(g^\top\delta b_{j,t})^2 \right]dt.
\]
Since
\[
(g^\top\delta b_{j,t})^2
\le \|g\|^2\|\delta b_{j,t}\|^2,
\]
together with
\[
\|g\| \le \frac{K_p}{\|\delta x\|},
\qquad \|\delta b_{j,t}\|
\le L_{b_j}\|\delta x\|,
\]
we obtain
\[
(g^\top\delta b_{j,t})^2
\le K_p^2L_{b_j}^{\,2}.
\]
Substituting into the variance expression yields
\eqref{eq:variance_bound}.
\end{proof}

\begin{theorem}[Finite-Time Transient Stability Decomposition]
\label{thm:transient_decomposition}

Under the assumptions of
Theorem~\ref{thm:mean_bound}
and Proposition~\ref{prop:variance},
\[
d\ln\|\delta X_t\|_p=dM_t+dN_t,
\]
where
\[
\begin{aligned}
dM_t&:=
\left(
\delta f_t^\top g
+\frac12\sum_{j=1}^{m}
\delta b_{j,t}^{\top}\nabla^2V\,\delta b_{j,t}
\right)dt,\\
dN_t&:=
\sum_{j=1}^{m}
g^\top\delta b_{j,t}\,dW_t^j,
\end{aligned}
\]
with
\[
\mathbb{E}[dM_t]\le\alpha\,dt,\qquad
\operatorname{Var}(dN_t)\le\beta\,dt,
\]
where
\begin{equation}
\alpha=
\mu_p(f)+
\frac{C_p}{2}\sum_{j=1}^{m}L_{b_j}^{\,2},
\label{eq:alpha}
\end{equation} and

\begin{equation}
\beta=
K_p^2\sum_{j=1}^{m}L_{b_j}^{\,2}.
\label{eq:beta}
\end{equation}

\end{theorem}

\begin{proof}
Equation~\eqref{eq:ito_log_norm}
immediately yields
\[
d\ln\|\delta X_t\|_p=dM_t+dN_t.
\]
The bounds
$\mathbb E[dM_t]\le\alpha dt$
and
$\operatorname{Var}(dN_t)\le\beta dt$
follow directly from
Theorem~\ref{thm:mean_bound}
and Proposition~\ref{prop:variance},
respectively.
\end{proof}

% \begin{theorem}[Finite-Time Transient Stability Decomposition]
% \label{thm:transient_decomposition}

% Under the assumptions of Theorem~\ref{thm:mean_bound} and Proposition~\ref{prop:variance}, the logarithmic perturbation dynamics satisfy
% \[
% d\ln\|\delta X_t\|_p
% =
% dM_t+dN_t,
% \]

% where \(dM_t\) denotes the deterministic drift
% contribution and \(dN_t\) denotes the stochastic
% martingale contribution.

% \[
% dM_t :=
% \left(
% \delta f_t^\top g
% +
% \frac12
% \sum_{j=1}^{m}
% \delta b_{j,t}^{\top}
% \nabla^2V
% \,
% \delta b_{j,t}
% \right)dt,
% \]

% and \[ dN_t :=
% \sum_{j=1}^{m}
% g^\top
% \delta b_{j,t}
% \,dW_t^j. \]
% Moreover, \( \mathbb{E}[dM_t] \le \alpha\,dt, \)
% with \(
% \alpha = \mu_p(f) +
% \frac{C_p}{2}
% \sum_{j=1}^{m}
% L_{b_j}^{\,2},
% \) and \( \operatorname{Var}(dN_t) \le \beta\,dt, \)
% where
% \( \beta = K_p^{2}
% \sum_{j=1}^{m} L_{b_j}^{\,2}.
% \)
% \end{theorem}
% \begin{proof}
% Equation~\eqref{eq:ito_log_norm} naturally separates
% the logarithmic perturbation dynamics into a drift
% term and a stochastic It\^o martingale term,

% \[
% d\ln\|\delta X_t\|_p
% =
% dM_t+dN_t.
% \]

% The expectation of the drift component satisfies
% \(
% \mathbb{E}[dM_t] \le \alpha\,dt,
% \)

% by Theorem~\ref{thm:mean_bound}, whereas the
% quadratic variation of the martingale component
% yields

% \[
% \operatorname{Var}(dN_t)
% \le
% \beta\,dt,
% \]

% by Proposition~\ref{prop:variance}.
% Hence the logarithmic perturbation dynamics admit the
% claimed decomposition into deterministic contraction
% and stochastic variability, completing the proof.
% \end{proof}

The quantities $\alpha$ and $\beta$ provide first- and
second-moment characterizations of finite-time logarithmic
perturbation growth and form the basis for the
probabilistic analysis in the next section.
%=====================================================
%=====================================================	

%=====================================================
%=====================================================	
\section{Probabilistic Transient Stability Bounds}
\label{sec:probabilistic}
This section derives finite-time probabilistic bounds on
transient amplification by combining the mean and variance
bounds of Section~\ref{sec:transient} with exponential
concentration arguments.

Define the logarithmic perturbation increment
\[
Y_t^{\Delta t}
= \ln\|\delta X_{t+\Delta t}\|_p
- \ln\|\delta X_t\|_p,
\]
over the interval $\Delta t$.

For sufficiently small $\Delta t$,
Theorem~\ref{thm:mean_bound} and Proposition~\ref{prop:variance}
imply
\begin{align}
\mathbb{E}[Y_t^{\Delta t}]
&\le
\alpha\,\Delta t \label{eq:mean_Y_new}
\\
\operatorname{Var}(Y_t^{\Delta t})
&\le
\beta\,\Delta t \label{eq:var_Y_new}
\end{align}
where $\alpha$ and $\beta$ are defined in
Theorem~\ref{thm:transient_decomposition}.
%-----------------------------------------------------
\subsection{Chernoff Bound}

For any $\zeta\ge0$, Chernoff's inequality yields
\begin{equation}
\mathbb{P}\!\left(Y_t^{\Delta t} \ge 0\right)
\le
\mathbb{E}\!\left[e^{\zeta Y_t^{\Delta t}}\right],
\label{eq:chernoff_start}
\end{equation}

For sufficiently small $\Delta t$, a second-order cumulant expansion yields
\begin{equation}
\log \mathbb{E}\!\left[e^{\zeta Y_t^{\Delta t}}\right]
\le
\zeta\,\mathbb{E}[Y_t^{\Delta t}]
+
\frac{\zeta^2}{2}\operatorname{Var}(Y_t^{\Delta t})
+ o(\Delta t),
\label{eq:cumulant_expansion}
\end{equation}
as $\Delta t \to 0$.

Substituting \eqref{eq:mean_Y_new} and \eqref{eq:var_Y_new}
into \eqref{eq:cumulant_expansion} yields
\begin{equation}
\log \mathbb{E}\!\left[e^{\zeta Y_t^{\Delta t}}\right]
\le
\zeta\,\alpha\,\Delta t
+
\frac{\zeta^2}{2}\beta\,\Delta t
+ o(\Delta t).
\end{equation}

Therefore,
\begin{equation}
\mathbb{P}\!\left(Y_t^{\Delta t} \ge 0\right)
\le \exp\!\left(
\zeta\,\alpha\,\Delta t
+ \frac{\zeta^2}{2}\beta\,\Delta t
+ o(\Delta t) \right).
\label{eq:chernoff_bound}
\end{equation}

For $\alpha \le 0$, the bound is minimized by choosing
\(
\zeta^\star = -\frac{\alpha}{\beta}.
\)

Substituting the optimal Chernoff parameter gives

\begin{equation}
\mathbb{P}\!\left(Y_t^{\Delta t} \ge 0\right)
\le \exp\!\left(
-\frac{\alpha^2}{2\beta}\,\Delta t
+ o(\Delta t) \right).
\label{eq:final_probability_bound}
\end{equation}
Equation~\eqref{eq:final_probability_bound}
shows that the probability of transient amplification
decays exponentially as deterministic contraction
increasingly dominates diffusion-induced variability. Consequently, $\alpha^2/\beta$ naturally emerges as the transient-risk index.
%-----------------------------------------------------
%-----------------------------------------------------
\subsection{Pathwise Safety}

Equation~\eqref{eq:final_probability_bound}
shows that mean transient stability does not guarantee
finite-time pathwise safety. Although the
expected logarithmic growth rate may be negative,
diffusion-induced variability permits transient
amplification with nonzero probability over finite
time horizons.

Collectively, these results characterize finite-time
transient amplification through the deterministic
growth-rate bound $\alpha$ and stochastic variability
bound $\beta$, forming the basis for the transient-risk
framework developed in Section~\ref{sec:design}.
%=====================================================
%=====================================================	

%=====================================================
%=====================================================	
\section{Data-Constrained and Projected Stochastic Dynamics}
\label{sec:projected}
This section extends the proposed framework to
projection-based stochastic systems.
%-----------------------------------------------------
\subsection{Projected Stochastic Dynamics}
Many estimation and navigation algorithms incorporate
continuous measurement updates or constraint
enforcement. 
We consider the projected stochastic dynamics
Let \( h(x,t) \in \mathbb{R}^k \)
denote a continuously differentiable observation or constraint function, with Jacobian
\(
H(x,t) := \frac{\partial h(x,t)}{\partial x}.
\)
Assume that $H(x,t)H(x,t)^{\top}$ is invertible for all $(x,t)$ of interest. Define the projection operator
\begin{equation}
\Pi(x,t) := I - 
H(x,t)^{\top}\!\big(H(x,t) H(x,t)^{\top}\big)^{-1} H(x,t).
\label{eq:projection}
\end{equation}

\begin{remark}
The invertibility of $H(x,t)H(x,t)^\top$ is a standard regularity condition in projection-based filtering and constrained estimation \cite{CrassidisJunkins2012,ReichCotter2015}.
\end{remark}

Consider the projected stochastic differential equation
\begin{align}
dX_t &=
\Pi(X_t,t)\, f(X_t,t)\, dt \nonumber\\
&\quad + H(X_t,t)^{\top}
\big(H(X_t,t) H(X_t,t)^{\top}\big)^{-1} \nonumber\\
&\quad \Big(dy_t^{\mathrm{obs}}
-\partial_t h(X_t,t)\, dt - dW_t
\Big),
\label{eq:projected_sde}
\end{align}
where $y_t^{\mathrm{obs}}$ denotes observed data and $W_t$ is a standard Wiener process.
%-----------------------------------------------------
\subsection{Perturbation Dynamics under Projection}

Let $\delta X_t$ denote a finite perturbation of~\eqref{eq:projected_sde}. 
The corresponding perturbation dynamics are given by
\begin{align}
d(\delta X_t)
&=
\delta(\Pi f)_t\, dt
\nonumber\\
&\quad
+ \delta\!\left(H^\top (H H^\top)^{-1}\right)_t
\left(
dy_t^{\mathrm{obs}} - \partial_t h\, dt - dW_t \right),
\label{eq:projected_perturbation}
\end{align}

where all derivatives are evaluated along the sample path $X_t$.
Here, $\delta(\Pi f)_t$ denotes the finite-difference increment of the composite mapping $x \mapsto \Pi(x,t)f(x,t)$, interpreted in the Lipschitz sense introduced in Section~\ref{sec:lognorm}.

The observation noise is assumed independent of the
process noise, and perturbations are interpreted in
the finite-difference sense.
%-----------------------------------------------------
\subsection{Logarithmic Norm Bound under Projection}
Applying the transient-growth analysis of Sections~\ref{sec:lognorm}--\ref{sec:probabilistic} to the composite mapping
\[
x \mapsto \Pi(x,t)f(x,t) + \mathcal{C}(x,t)\big(\bar{y}_t^{\mathrm{obs}} - \partial_t h(x,t)\big),
\]
yields the following bound.

\begin{theorem}[Projected Transient Growth Bound]
\label{thm:projected_bound}
The expected logarithmic perturbation growth rate satisfies the bound
\begin{align}
\mathbb{E}
\!\left[ d\ln\|\delta X_t\|_p \right]
&\le \mu_p \!\Big( \Pi f + \mathcal{C}
\big( \bar y_t^{\mathrm{obs}} -
\partial_t h \big) \Big) dt \nonumber\\
&\quad + \frac{C_p}{2} \sum_{j}
L_{\mathcal{C}^{j}}^{\,2} \,dt .
\label{eq:projected_mean_bound}
\end{align}
where
$\mathcal C(x,t) = H(x,t)^{\top}\!\big(H(x,t)H(x,t)^{\top}\big)^{-1}$
is the data-injection operator,
$\mathcal C^j$ its $j$th column, and
$\bar y_t^{\rm obs}$ the observation mean rate
interpreted in the weak (distributional) sense.
\end{theorem}

\begin{proof}
The projected perturbation dynamics
\eqref{eq:projected_perturbation}
have effective drift
\[
f_{\rm proj}
= \Pi f + \mathcal C
\big( \bar y_t^{\rm obs} - \partial_t h \big),
\]
and diffusion fields given by the columns
$\mathcal C^j$ of $\mathcal C$.
Applying Theorem~\ref{thm:mean_bound}
to these dynamics immediately yields
\eqref{eq:projected_mean_bound}.
\end{proof}

Define the projected transient-growth parameter
\[
\alpha_{\rm proj} := \mu_p \!\Big( \Pi f
+ \mathcal C ( \bar y_t^{\rm obs} -
\partial_t h) \Big) +
\frac{C_p}{2}
\sum_{j=1}^{m}
L_{\mathcal C^j}^{\,2},
\]
%-----------------------------------------------------
\subsection{Impact of Data Injection on Transient Stability}

\begin{proposition}[Data-Induced Variance Growth]
\label{prop:data_variance}

The infinitesimal variance growth of the logarithmic
perturbation norm satisfies

\begin{equation}
\operatorname{Var}
\!\left(
d\ln\|\delta X_t\|_p
\right)
\le
K_p^2
\sum_{j=1}^{m}
L_{\mathcal C^j}^{\,2}
\,dt ,
\label{eq:data_variance}
\end{equation}

where
$L_{\mathcal C^j}$
denotes the Lipschitz constant of the
$j$th column of the data-injection operator
$\mathcal C$.

\end{proposition}

\begin{proof}
From
\eqref{eq:projected_perturbation},
the stochastic contribution is
\[
\sum_{j=1}^{m} \delta(\mathcal C^j)_t^\top g\,dW_t^j,
\]
where
\[
g=\nabla_{\delta x}\ln\|\delta x\|_p.
\]
Using It\^o isometry,
\[
\operatorname{Var}
\!\left( d\ln\|\delta X_t\|_p
\right) = \sum_{j=1}^{m}
\mathbb E \!\left[ (\delta(\mathcal C^j)_t^\top g)^2
\right]dt.
\]
Since
\[
\|g\| \le \frac{K_p}{\|\delta x\|},
\qquad \|\delta(\mathcal C^j)_t\|
\le L_{\mathcal C^j}\|\delta x\|,
\]
we obtain
\[
(\delta(\mathcal C^j)_t^\top g)^2
\le K_p^2L_{\mathcal C^j}^{\,2},
\]
which yields
\eqref{eq:data_variance}.
\end{proof}

The projected diffusion-variability parameter
\[
\beta_{\rm proj} := 
K_p^2 \sum_{j=1}^{m} L_{\mathcal C^j}^{\,2},
\]

\begin{theorem}[Projected Transient Stability]
\label{thm:projected_characterization}
The finite-time transient behavior of the projected
stochastic system is characterized by the pair
$(\alpha_{\rm proj},\beta_{\rm proj})$.
\end{theorem}
\begin{proof}
The result follows directly from
Theorem~\ref{thm:projected_bound} and
Proposition~\ref{prop:data_variance}.
\end{proof}

These results extend the proposed transient-stability
framework to projection-based stochastic systems and
provide the basis for the numerical validation in the
next section.
%=====================================================
%=====================================================

%=====================================================
%=====================================================	
\section{Design Implications for Transient-Risk-Aware Systems}
\label{sec:design}
\subsection{Transient-Risk Index}

The probabilistic analysis of
Section~\ref{sec:probabilistic}
shows that finite-time transient amplification is
governed by the ratio
$\alpha^{2}/\beta$.
This observation motivates the definition of the
transient-risk index

\begin{equation}
J_{\mathrm{risk}} := \frac{\alpha^{2}}{\beta},
\label{eq:risk_metric}
\end{equation}
whenever
$\beta>0$.

\begin{proposition}[Monotonicity of the Transient-Risk Index]
\label{prop:risk_ordering}

Consider two stochastic systems with transient-risk
indices
\[
J_i=\frac{\alpha_i^2}{\beta_i},
\qquad i=1,2,
\]
defined by \eqref{eq:risk_metric}, and suppose that
\(J_1>J_2\). Then, for sufficiently small
\(\Delta t>0\),
\[
\exp\!\left(
-\frac12J_1\Delta t
\right)
<
\exp\!\left(
-\frac12J_2\Delta t
\right).
\]
Consequently, the Chernoff upper bound
\eqref{eq:final_probability_bound}
for transient amplification is strictly smaller for
System~1 than for System~2.

\end{proposition}

\begin{proof}
The result follows immediately from
\eqref{eq:final_probability_bound}, since the
exponential function is strictly increasing and
$J_1>J_2$ implies
\[
\exp\!\left(-\tfrac12J_1\Delta t\right)
<
\exp\!\left(-\tfrac12J_2\Delta t\right).
\]
\end{proof}
Proposition~\ref{prop:risk_ordering}
shows that larger $J_{\rm risk}$
correspond to smaller theoretical upper bounds on
finite-time transient amplification, making
$J_{\rm risk}$ a natural design objective.
%----------------------------------------------------------
\subsection{Design Optimization}

The transient-risk index provides a natural objective
for finite-time robust design by explicitly balancing
deterministic contraction and diffusion-induced
variability,
\[
\max J_{\rm risk} = \max\frac{\alpha^2}{\beta},
\]
subject to mission or system constraints.
Maximizing $J_{\mathrm{risk}}$ minimizes the theoretical
upper bound on finite-time transient amplification
while balancing deterministic contraction and
diffusion-induced variability.

Optimization variables may include guidance and
control gains, navigation filter parameters,
trajectory design variables, sensor configurations,
projection operators, and data-assimilation
strategies. The transient-risk index therefore extends the
framework from transient analysis to
transient-risk-aware design.

%-----------------------------------------------------
\subsection{Extension to Projected Systems}

The projected transient stability theory of
Section~\ref{sec:projected} extends the transient-risk formulation to continuously corrected stochastic systems. Replacing the unconstrained parameters $\alpha$ and $\beta$ by their projected counterparts, \( \alpha_{\rm proj}, \beta_{\rm proj}, \) yields the projected transient-risk index
\[
J_{\rm proj} =
\frac{\alpha_{\rm proj}^{2}}
{\beta_{\rm proj}},
\]
which quantifies finite-time transient robustness
under continuous measurement updates and projection
operations. This provides a principled basis for designing
projection operators and data-assimilation
strategies that balance estimation performance and
transient robustness.

More generally, the proposed framework applies to
nonlinear stochastic systems satisfying the
assumptions of
Sections~\ref{sec:preliminaries}--
\ref{sec:projected}, independent of the application
domain. Representative applications include spacecraft
guidance, navigation, and control, powered-descent
guidance, planetary landing, nonlinear state
estimation, autonomous rendezvous and docking,
formation flying, and other safety-critical
autonomous aerospace systems.

The transient-risk index also provides a quantitative
criterion for autonomous guidance and decision
making. When multiple feasible actions are available, candidate strategies can be ranked according to their
transient-risk indices together with conventional
performance measures such as fuel consumption,
flight time, landing accuracy, mission safety. This enables transient-risk-aware guidance and trajectory selection by favoring actions with lower theoretical probabilities of finite-time transient
amplification while satisfying conventional mission
performance objectives.

\section{Discussion and Aerospace Design Implications}
\label{sec:discussion}

The proposed finite-amplitude transient stability
framework extends classical stochastic stability
analysis from asymptotic and infinitesimal
perturbation behavior to finite-time transient
amplification in nonlinear stochastic systems.
The theoretical developments together with the
numerical and telemetry validation provide several
insights for aerospace guidance, navigation, and
control systems operating under stochastic
uncertainty.

%-----------------------------------------------------
\subsection{Implications for Aerospace Guidance and Navigation}

Mission-critical aerospace guidance, navigation,
and control operations, including atmospheric entry,
powered descent, planetary landing, rendezvous, and
autonomous docking, require reliable operation over
finite time horizons in the presence of uncertainty.
Although conventional stability analyses ensure
asymptotic or local convergence, significant
transient perturbation growth may still occur before
asymptotic behavior dominates. Such transient
amplification can degrade navigation accuracy,
increase guidance errors, and reduce mission
robustness during safety-critical operations.

Modern aerospace navigation systems continuously
fuse measurements from inertial sensors, vision
systems, terrain-relative navigation, GNSS (where
available), and other onboard sensors. These
measurement updates modify both deterministic
contraction properties and diffusion-induced
variability. Consequently, projection and
data-assimilation strategies directly influence
finite-time transient robustness.

The proposed framework provides a quantitative
characterization of short-horizon transient behavior
through finite-amplitude perturbation evolution,
thereby complementing classical Lyapunov- and
Jacobian-based analyses with information that is
directly relevant to operational decision making.

%-----------------------------------------------------
\subsection{Engineering Design Implications}

The analytical results naturally suggest several
practical design principles for stochastic
autonomous aerospace systems.

\begin{itemize}

\item maximize deterministic contraction by reducing
the transient-growth parameter $\alpha$;

\item reduce diffusion-induced variability by
minimizing the variance parameter $\beta$ through
improved sensing, estimation, and data-assimilation
strategies;

\item design projection operators that improve
measurement consistency while limiting transient
amplification;

\item monitor empirical transient-growth and
diffusion-variability indicators online to support
transient-risk-aware supervision and fault
detection.

\end{itemize}

These principles provide a quantitative basis for
designing guidance, navigation, estimation, and
control algorithms that explicitly account for
finite-time robustness in addition to conventional
asymptotic stability.

%-----------------------------------------------------
\subsection{Relation to Existing Stability Frameworks}

The proposed methodology complements existing
Lyapunov, incremental stability, and contraction
theories by explicitly characterizing
finite-amplitude transient amplification over finite
operational horizons.

Whereas classical approaches primarily quantify
asymptotic convergence or infinitesimal perturbation
behavior, the present framework directly evaluates
finite-amplitude perturbation evolution, providing
additional insight into short-term robustness under
stochastic uncertainty.

Rather than replacing existing stability tools, the
proposed methodology provides an additional
analytical tool for evaluating transient robustness
in nonlinear stochastic systems.

%-----------------------------------------------------
\subsection{Illustrative Extension to Autonomous Spacecraft Guidance}

Consider an autonomous spacecraft performing powered
descent or planetary landing. Let $ \Gamma=\{\Gamma_1,\Gamma_2,\ldots,\Gamma_N\}$
denote the set of feasible guidance trajectories
generated by an onboard guidance algorithm.

The stochastic dynamics associated with each
candidate trajectory are represented by the
nonlinear It\^o model introduced in
Section~\ref{sec:preliminaries}, where the drift and
diffusion fields depend on spacecraft dynamics,
terrain uncertainty, sensing geometry, and
environmental disturbances.

Applying the proposed framework yields the
transient-growth parameter
$\alpha(\Gamma_i)$,
the diffusion-variability parameter
$\beta(\Gamma_i)$,
and the corresponding transient-risk index

\[
J_{\rm risk}(\Gamma_i)
=
\frac{\alpha(\Gamma_i)^2}
{\beta(\Gamma_i)}.
\]

Candidate trajectories can therefore be ranked using
their transient-risk indices together with
conventional mission objectives such as fuel
consumption, flight time, and landing accuracy.
Consequently, the proposed framework provides a
quantitative basis for transient-risk-aware guidance
and trajectory selection in future autonomous
spacecraft missions.
%=====================================================

%=========================================================
%=========================================================
\section{Monte-Carlo Validation of Finite-Amplitude Transient-Growth Prediction}
\label{sec:mc_validation}
%=========================================================
This section validates the proposed finite-amplitude
logarithmic measure using Monte Carlo
realizations by comparing its transient-growth
prediction with the classical Jacobian logarithmic
measure. The benchmark nonlinear system is
\begin{equation}
\begin{aligned}
\dot{x}_1 &= -x_1 + 2x_2 - x_1x_2, \\
\dot{x}_2
&= -2x_2 + x_1^3 - x_1^2x_2 ,
\end{aligned}
\label{eq:mc_system}
\end{equation}

which exhibits strong nonlinear transient
amplification and therefore provides a challenging
benchmark.
For each realization, the initial condition was sampled
uniformly from
$(x_1(0),x_2(0))\in[-1,1]^2$
with finite perturbation
$\delta x_0=10^{-2}[1\;0]^\top$.
The nominal and perturbed trajectories were propagated
over $t\in[0,20]$ using MATLAB \texttt{ode45}
(RelTol $10^{-10}$, AbsTol $10^{-12}$)
with a fixed random seed.

For each realization, the true logarithmic perturbation
growth rate was computed as

\begin{equation}
\dot G(t) =
\frac{d}{dt}
\ln \left(
\frac{\|\delta x(t)\|}
{\|\delta x_0\|}
\right),
\label{eq:true_growth_mc}
\end{equation}
where $\delta x(t)=x_p(t)-x_n(t)$.

The true growth rate was compared with
\begin{enumerate}
\item the classical Jacobian logarithmic norm $\mu_{\rm Jac}(t)=\mu_2(J(x(t)))$;

\item the proposed finite-amplitude logarithmic
measure $\mu_{\rm FA}(t)$.
\end{enumerate}

Prediction accuracy was quantified using the root mean
square error
\begin{equation}
{\rm RMSE}
=
\sqrt{
\frac{1}{N}
\sum_{k=1}^{N}
\big(
\dot G_k-\mu_k
\big)^2
},
\label{eq:rmse_metric}
\end{equation}
and the Pearson correlation coefficient between the
predicted and true growth rates.

Statistical significance was assessed using a paired
Wilcoxon signed-rank test.

%=========================================================
\subsection{Monte Carlo Statistical Validation}
%=========================================================
Across all Monte Carlo realizations, the proposed finite-amplitude logarithmic measure achieved lower prediction error than the classical Jacobian approximation in every realization. A paired Wilcoxon signed-rank test rejected equal predictive performance ($p = 3.3 \times 10^{-165}$), confirming that the observed improvement is statistically significant. These results demonstrate that directly evaluating finite-amplitude perturbation evolution provides substantially more accurate prediction of transient growth than classical Jacobian linearization, particularly when nonlinear interactions dominate the system dynamics.

\begin{figure}[t]
\centering
\includegraphics[width=0.70\linewidth]
{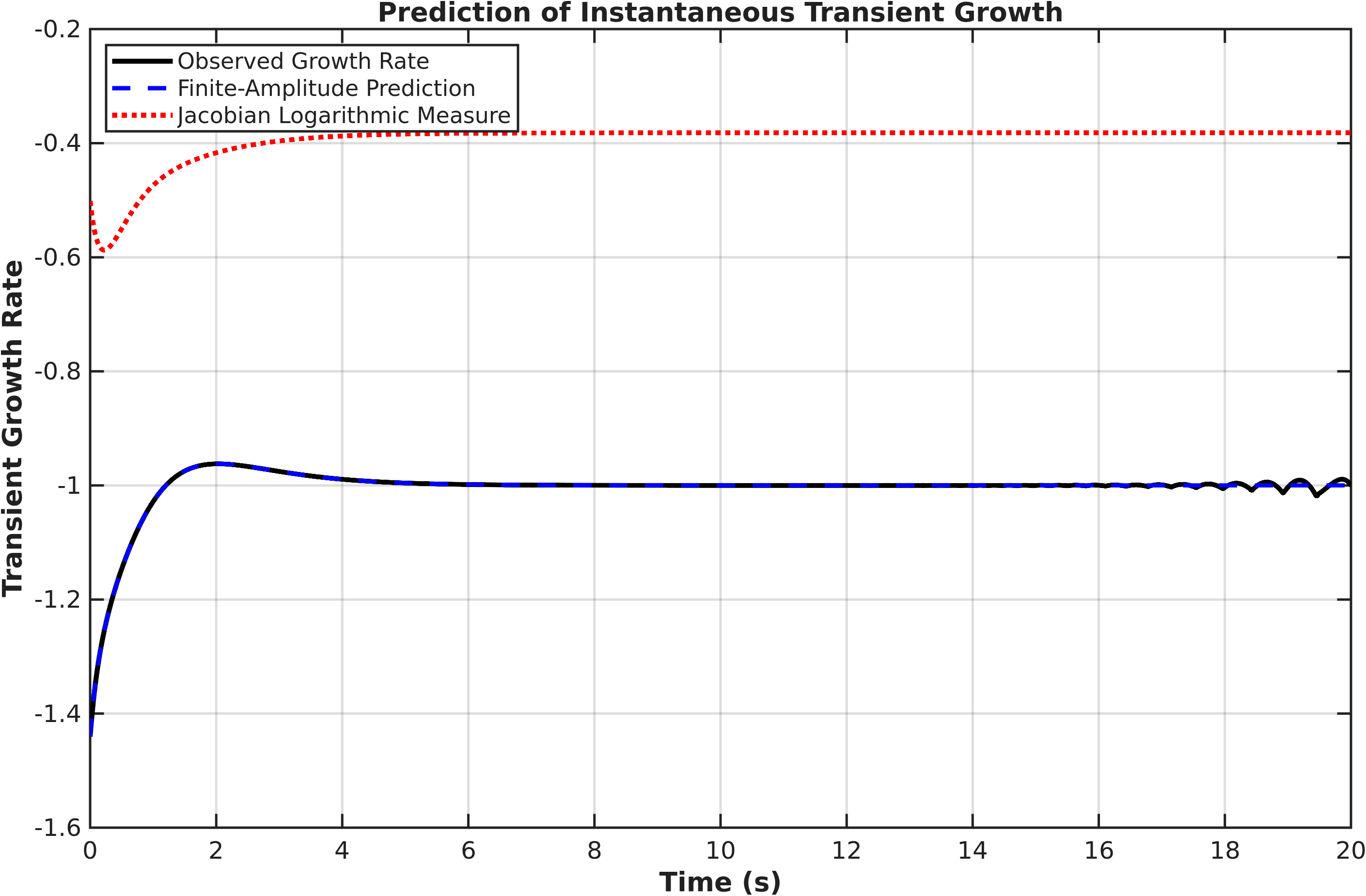}
\caption{
Comparison of the observed transient-growth rate with
predictions obtained using the proposed finite-amplitude
and classical Jacobian logarithmic measures for a
representative Monte Carlo realization.
}
\label{fig:mc_prediction}
\end{figure}

Figure~\ref{fig:mc_prediction}
shows that the proposed measure closely tracks the
observed transient-growth rate, whereas the Jacobian
approximation exhibits systematic prediction bias.

\begin{figure}[t]
\centering
\includegraphics[width=0.70\linewidth]
{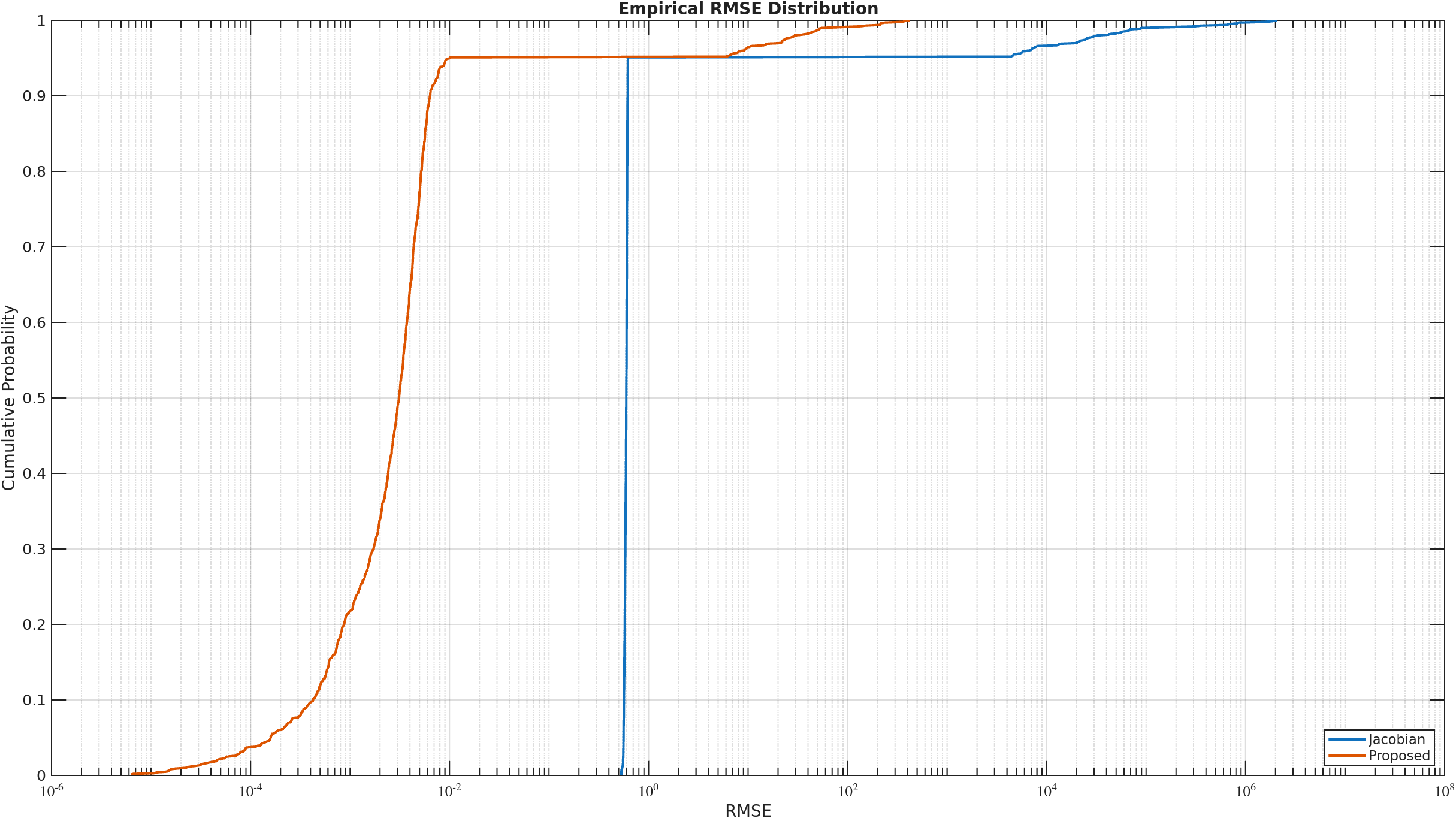}
\caption{
Distribution of prediction RMSE over Monte Carlo
realizations. The proposed finite-amplitude logarithmic
measure maintains consistently lower prediction error
than the classical Jacobian logarithmic measure.
}
\label{fig:mc_rmse}
\end{figure}

\begin{figure}[t]
\centering
\includegraphics[width=0.70\linewidth]
{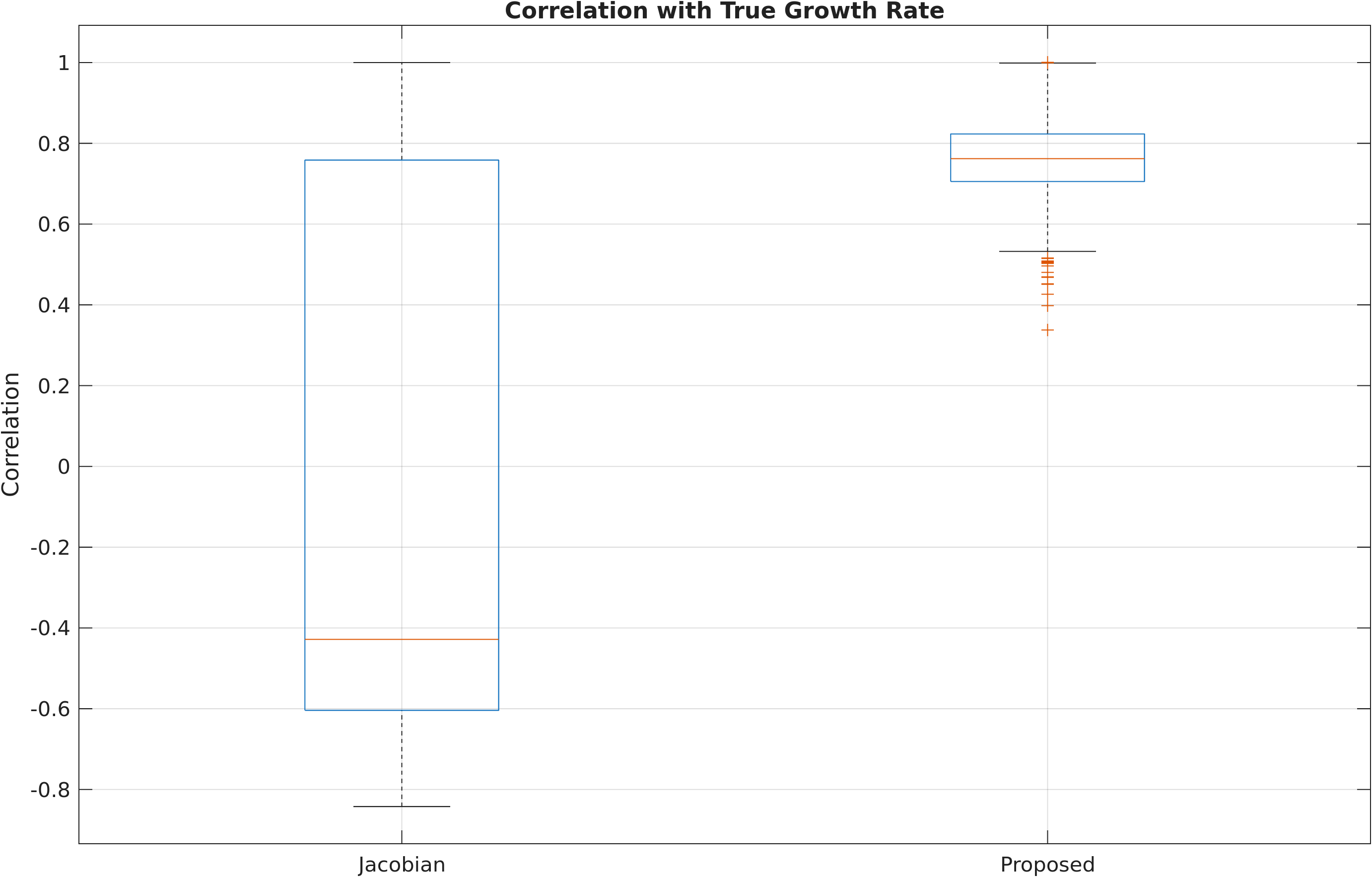}
\caption{
Correlation between predicted and observed
transient-growth rates over Monte Carlo
realizations for the proposed finite-amplitude and
classical Jacobian logarithmic measures.
}
\label{fig:mc_corr}
\end{figure}

\begin{figure}[t]
\centering
\includegraphics[width=0.70\linewidth]
{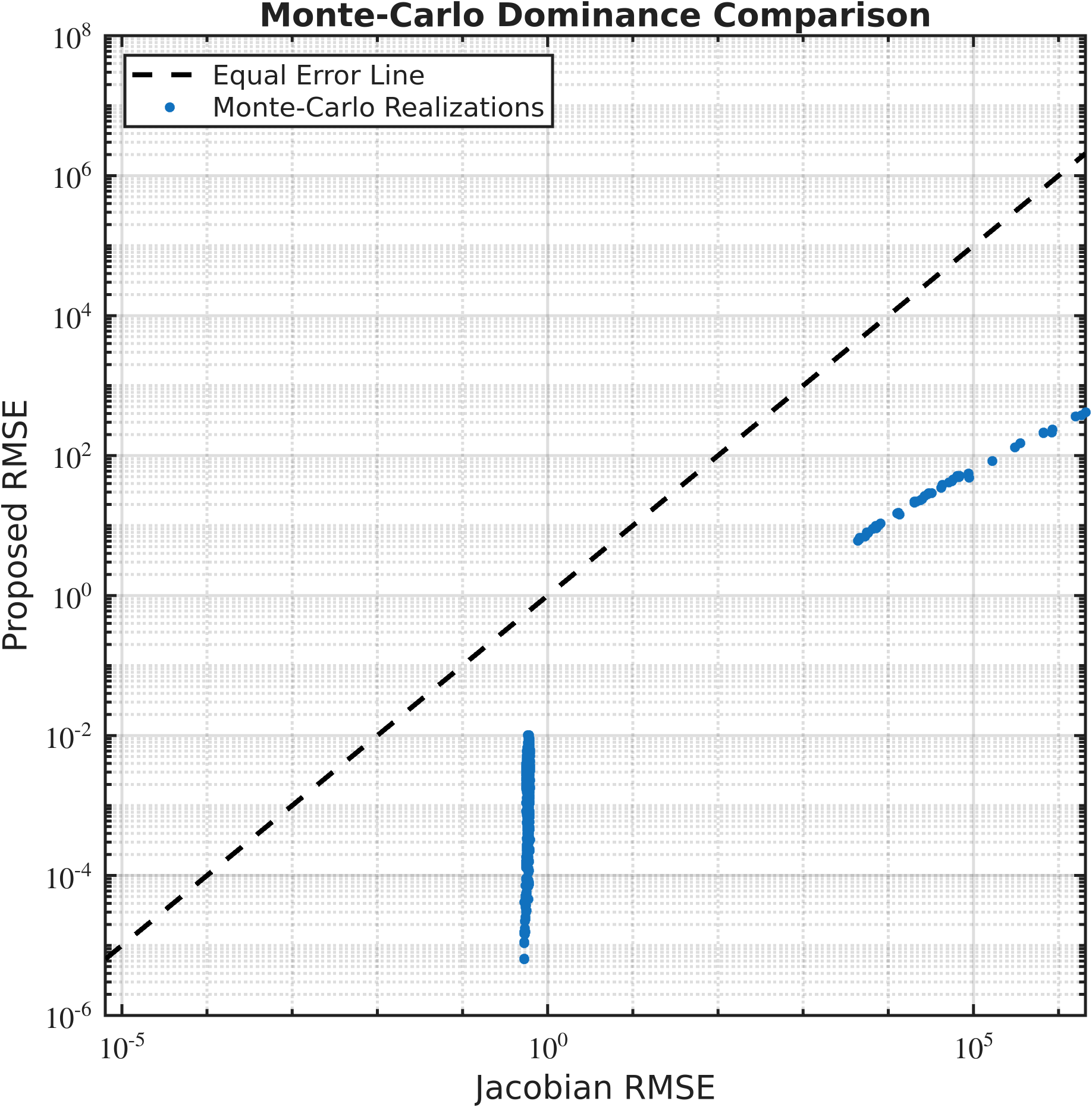}
\caption{
Monte Carlo comparison of prediction RMSE obtained
using the classical Jacobian logarithmic measure and
the proposed finite-amplitude logarithmic measure.
Each point represents one realization.
}
\label{fig:mc_dom}
\end{figure}

Collectively,
Figs.~\ref{fig:mc_rmse}--\ref{fig:mc_dom}
show that the proposed method consistently achieves
lower prediction error, stronger correlation with the
observed transient-growth rate, and uniformly
superior predictive performance across the Monte
Carlo ensemble.

\subsection{Summary of Quantitative Performance}
\label{sec:quantitative_validation}

Table~\ref{tab:mc_results}
summarizes the quantitative performance of the
proposed and classical logarithmic measures over Monte Carlo realizations.

\begin{table}[t]
\caption{
Quantitative comparison of transient-growth prediction
performance over 1000 Monte Carlo realizations.
}
\label{tab:mc_results}
\centering
\renewcommand{\arraystretch}{1.15}
\begin{tabular}{lcc}
\toprule
\textbf{Metric} & \textbf{Jacobian} & \textbf{Proposed} \\
\midrule
Median RMSE                  & 0.5977              & 0.00310 \\
Mean RMSE                    & $1.04\times10^{4}$  & 3.292 \\
Mean MAE                     & 234.518             & 0.0791 \\
Mean Maximum Error           & $4.76\times10^{5}$  & 153.635 \\
Median Spearman Correlation  & $-0.428$            & 0.762 \\
Mean Spearman Correlation    & 0.032               & 0.769 \\
Mean Pearson Correlation     & 0.424               & 0.990 \\
\midrule
Mean RMSE Reduction          & \multicolumn{2}{c}{99.97\%} \\
Wilcoxon Signed-Rank $p$-value & \multicolumn{2}{c}{$3.33\times10^{-165}$} \\
\bottomrule
\end{tabular}
\end{table}

Across 1000 Monte Carlo realizations, the proposed finite-amplitude logarithmic measure reduced the mean RMSE by 99.97\% and significantly outperformed the classical Jacobian approximation ($p = 3.33 \times 10^{-165}$), providing strong numerical evidence supporting the proposed finite-amplitude transient-growth framework.
%-----------------------------------------------------
%-----------------------------------------------------

%-----------------------------------------------------
%-----------------------------------------------------

\begin{figure}[t]
	\centering
	\includegraphics[width=0.70\linewidth]{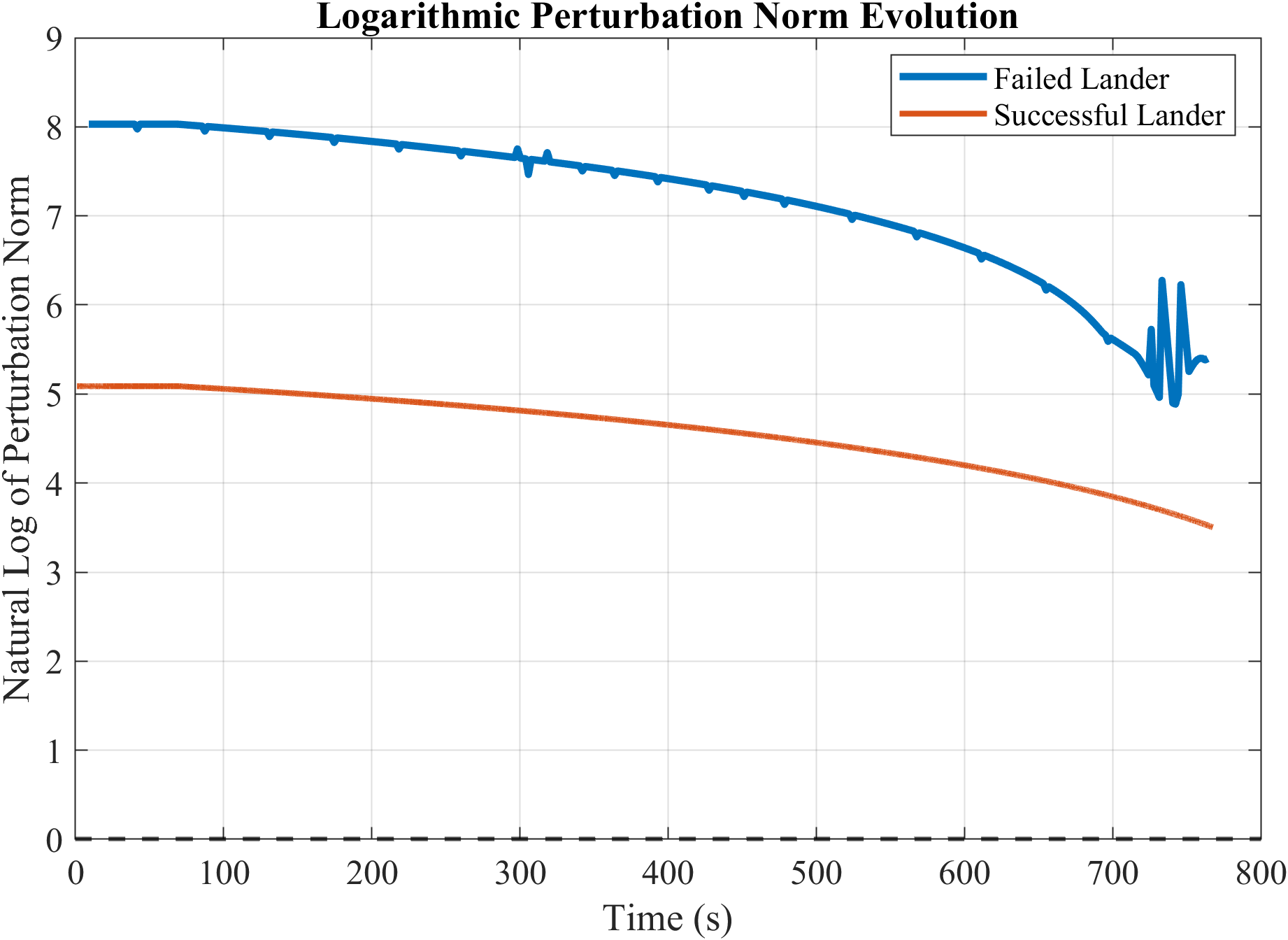}
	\caption{ Logarithmic perturbation norm
$\ln\|\delta X(t)\|_2$ during successful and failed lunar descent.}
	\label{fig:perturb}
\end{figure}

%-----------------------------------------------------
\begin{figure}[t]
	\centering
	\includegraphics[width=0.70\linewidth]{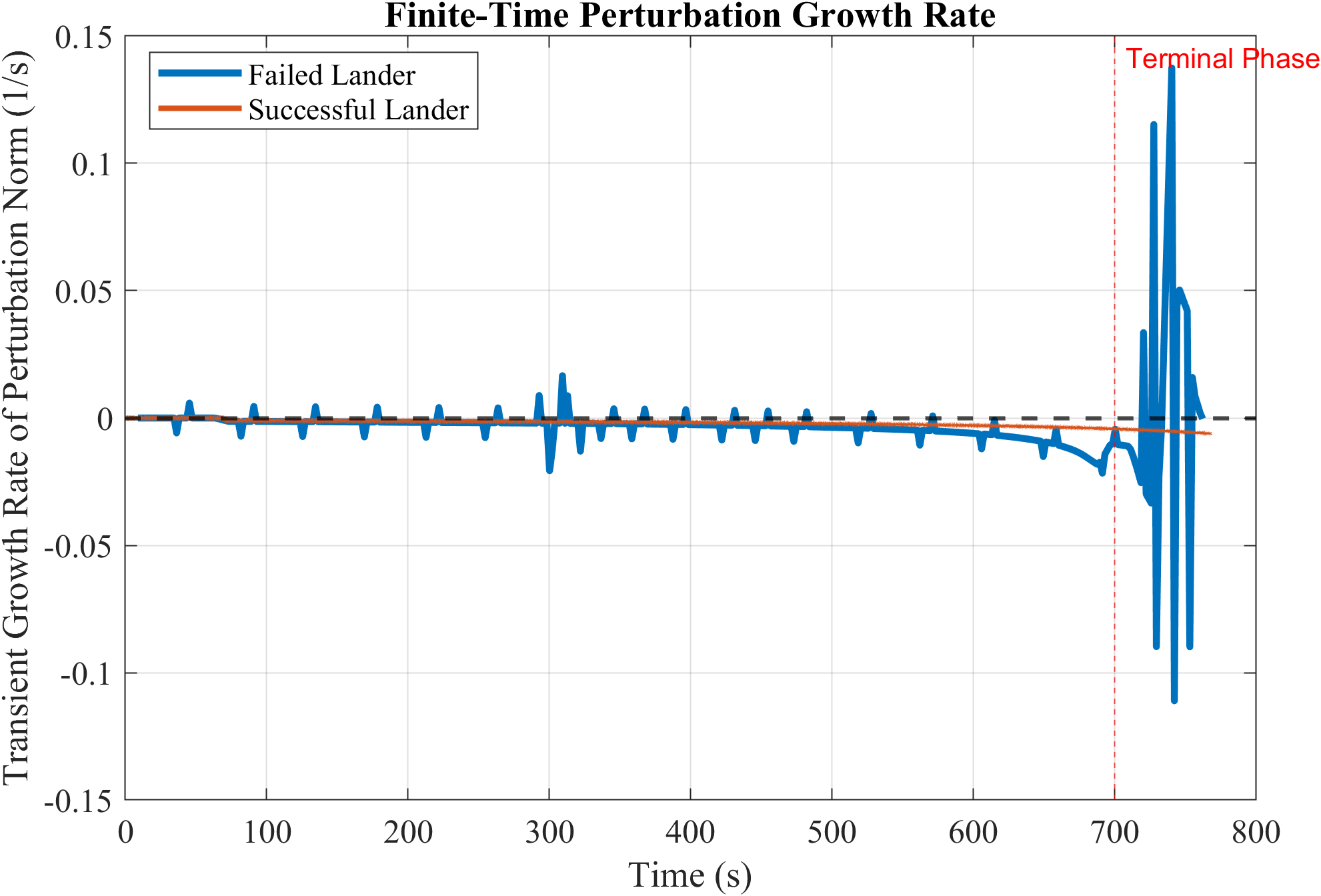}
	\caption{ Empirical transient-growth indicator
$\alpha(t)$ during successful and failed lunar descent. }
	\label{fig:alpha}
\end{figure}

%-----------------------------------------------------
\begin{figure}[t]
	\centering
	\includegraphics[width=0.70\linewidth]{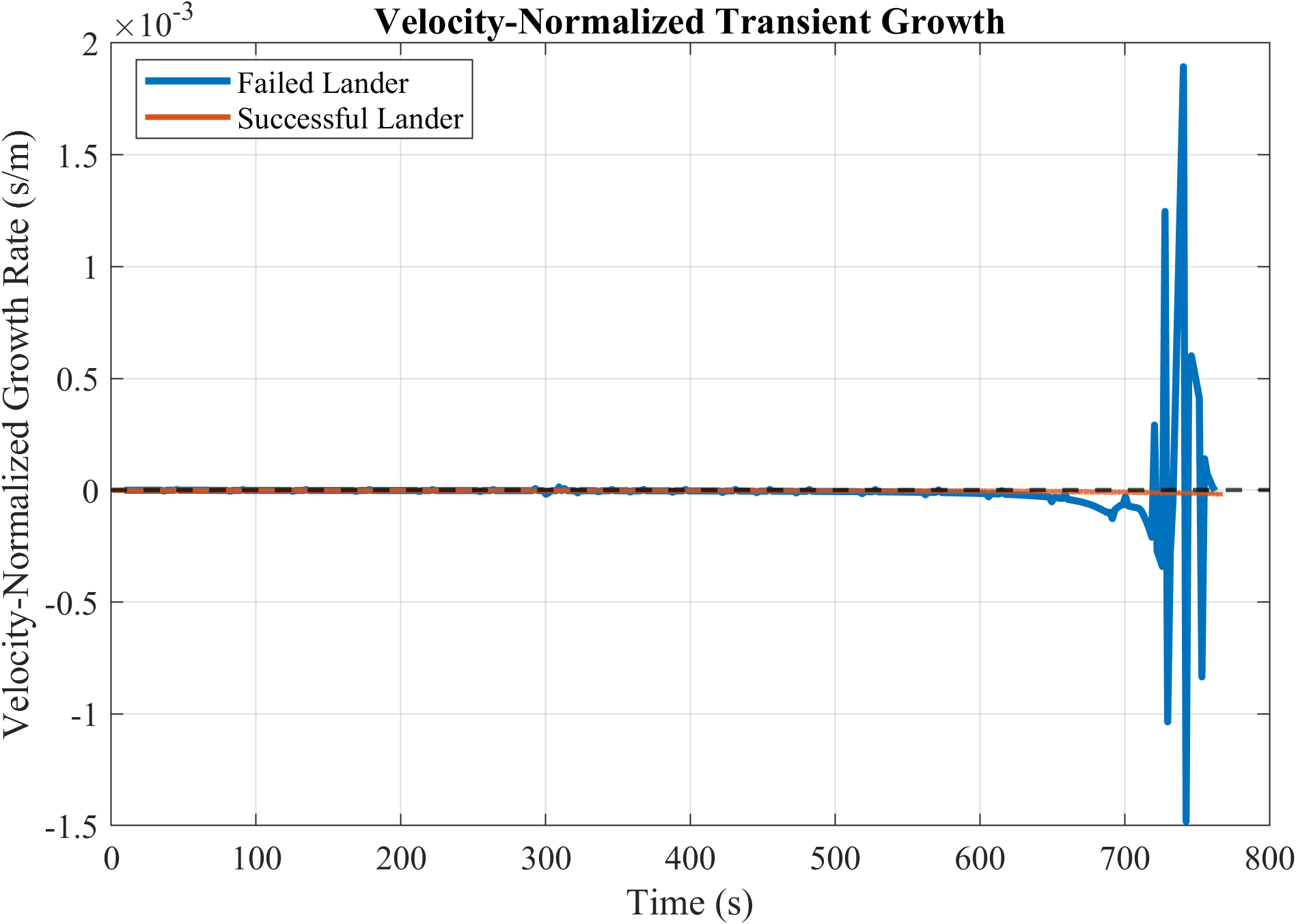}
	\caption{ Velocity-normalized empirical transient-growth indicator during successful and failed lunar descent. }
	\label{fig:alpha_norm}
\end{figure}

%-----------------------------------------------------
\begin{figure}[t]
	\centering
	\includegraphics[width=0.70\linewidth]{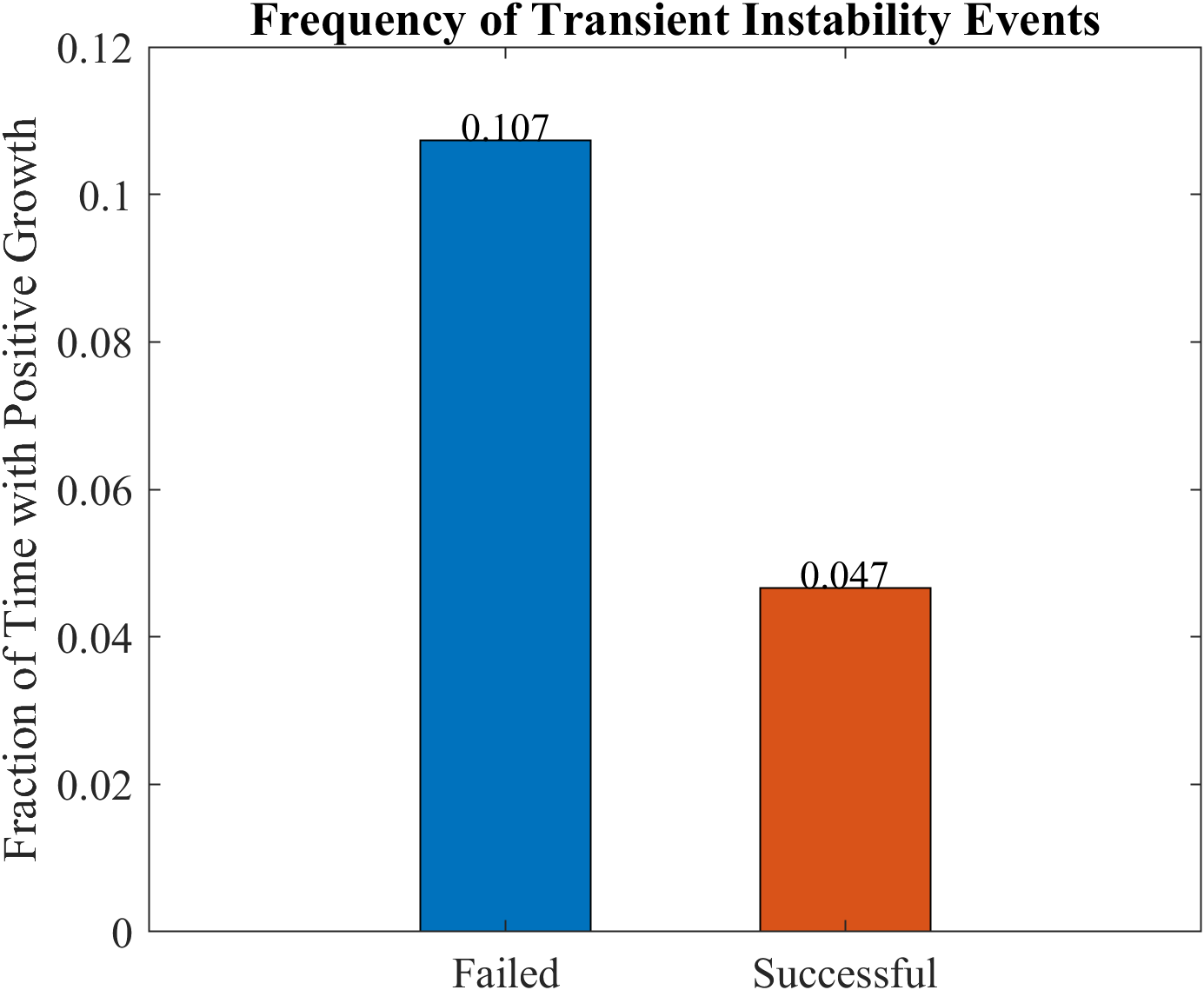}
	\caption{
Frequency of transient instability events,
defined as the fraction of time for which
$\alpha(t)>0$.
}
	\label{fig:freq}
\end{figure}

%-----------------------------------------------------
\begin{figure}[t]
	\centering
	\includegraphics[width=0.70\linewidth]{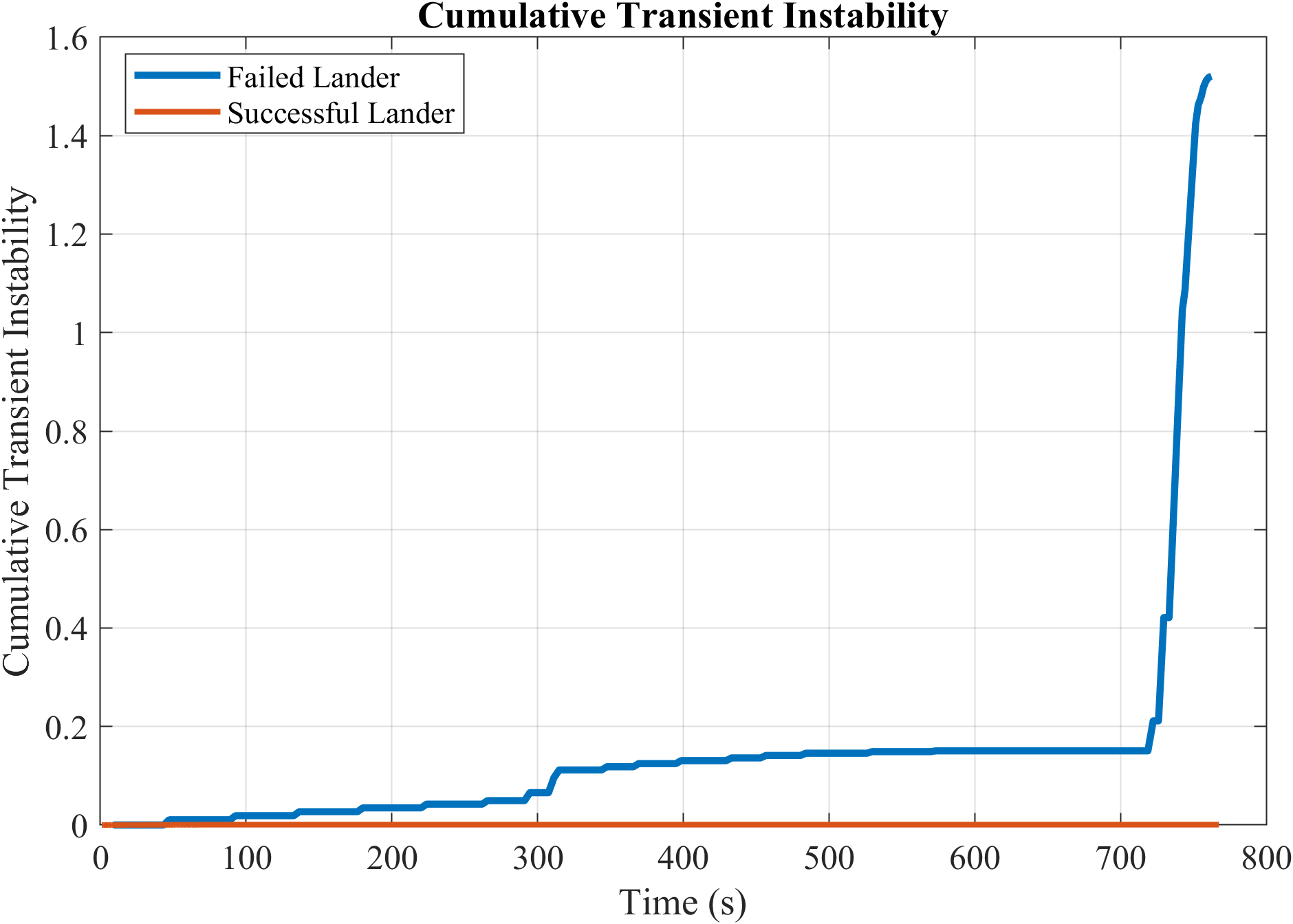}
	\caption{
Cumulative transient instability computed from the
empirical transient-growth indicator.
}
	\label{fig:cumulative}
\end{figure}
%-----------------------------------------------------
\begin{figure}[t]
	\centering
	\includegraphics[width=0.70\linewidth]{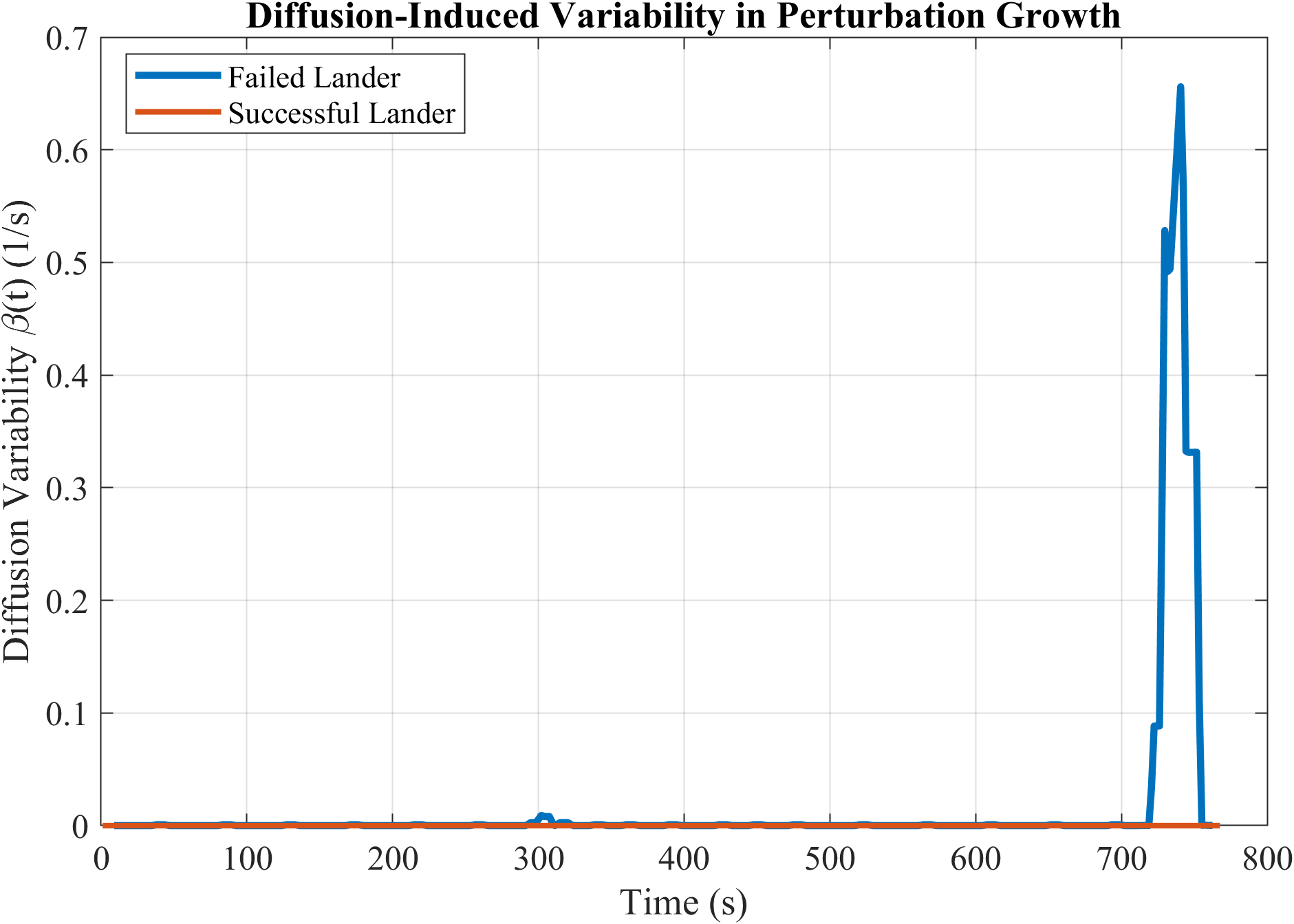}
	\caption{
Estimated diffusion-variability indicator
$\beta(t)$
during successful and failed lunar descent.
}
	\label{fig:beta}
\end{figure}
%-----------------------------------------------------
%-----------------------------------------------------
\section{Experimental Validation Using Flight-Like Lunar Descent Telemetry}
\label{sec:numerical}
This section experimentally evaluates the proposed
finite-amplitude transient stability framework using
flight-like lunar descent telemetry. The objective is
to determine whether the telemetry-derived transient
indicators consistently distinguish successful and
failed landing trajectories and provide experimental
support for the theoretical predictions developed in
Sections~\ref{sec:lognorm}--\ref{sec:projected}.
%-----------------------------------------------------
\subsection{Flight Telemetry and Validation Objective}

To establish empirical counterparts of the theoretical
quantities developed in
Sections~\ref{sec:lognorm}--\ref{sec:projected},
the following telemetry-derived observables are
constructed directly from the recorded descent
trajectories.

\begin{itemize}
\item Logarithmic perturbation norm \( \ln\|\delta X(t)\|_2, \) which characterizes the evolution of finite-amplitude perturbations.

\item Empirical transient-growth indicator
\[
\alpha(t)
\approx
\frac{
\ln\|\delta X(t+\Delta t)\|_2
-
\ln\|\delta X(t)\|_2
}
{\Delta t},
\]
which estimates the instantaneous logarithmic
perturbation growth rate.

\item Velocity-normalized transient-growth indicator,
used to distinguish intrinsic transient instability
from variations caused by changing descent velocity.

\item Diffusion-variability indicator
$\beta(t)$,
computed as a local estimate of the variance of
logarithmic perturbation increments over a moving
time window. This observable provides a
data-driven estimate of diffusion-induced variability
consistent with the variance term appearing in the
It\^o-based transient stability formulation.

\item Frequency of transient instability events,
defined by the fraction of time for which \( \alpha(t)>0. \)

\item Cumulative transient instability,

\[
\int_{0}^{t}
\max\!\left(\alpha(\tau),0\right)\,d\tau,
\]

which quantifies the accumulated exposure to
finite-time transient amplification.

\end{itemize}

\begin{remark}
The quantities $\alpha(t)$ and $\beta(t)$ used in this
section are empirical estimates of the instantaneous
transient-growth rate and diffusion-induced
variability, respectively. They serve as observable
counterparts of the theoretical coefficients
$\alpha$ and $\beta$ introduced in
Sections~\ref{sec:transient}--\ref{sec:probabilistic},
but are not identical to the analytical bounds defined
in \eqref{eq:alpha} and \eqref{eq:beta}.
\end{remark}
%-----------------------------------------------------
%-----------------------------------------------------
\subsection{Perturbation Growth and Transient Indicators}

Figures~\ref{fig:perturb} and~\ref{fig:alpha}
collectively illustrate the relationship between
finite perturbation evolution and the corresponding
empirical transient-growth indicator.

The logarithmic perturbation norm shown in
Fig.~\ref{fig:perturb} remains well behaved throughout
the successful descent, whereas the failed trajectory
exhibits localized perturbation amplification during
the terminal descent phase. This amplification is
accompanied by pronounced positive excursions of the
empirical transient-growth indicator $\alpha(t)$
(Fig.~\ref{fig:alpha}) over
$t\approx700$--$740~\mathrm{s}$,
indicating intervals of finite-time transient
instability.

The coincidence of perturbation amplification and
positive values of the empirical transient-growth
indicator is consistent with the theoretical prediction
that finite-time transient instability is associated
with localized perturbation growth.
%-----------------------------------------------------
%-----------------------------------------------------
\subsection{Diffusion Variability and Intrinsic Instability}
Figures~\ref{fig:alpha_norm} and ~\ref{fig:beta} together show that the transient instability observed during the failed descent is an intrinsic property of the underlying dynamics rather than a consequence of nominal trajectory variations.

The diffusion-variability indicator $\beta(t)$
(Fig.~\ref{fig:beta}) remains negligible throughout
the successful descent but increases markedly during
the terminal phase of the failed trajectory,
consistent with the predicted increase in
diffusion-induced variability accompanying
finite-time transient instability. A smaller rise
around $t\approx300~\mathrm{s}$ coincides with a
documented guidance--engine interaction anomaly,
indicating that the proposed observable is sensitive
to localized dynamical inconsistencies before mission
failure.

Figure~\ref{fig:alpha_norm} further shows that the
terminal excursions of the empirical transient-growth
indicator persist after normalization by descent
velocity. Thus, the observed transient amplification
cannot be explained solely by nominal descent-rate
variations but reflects the underlying nonlinear
dynamics. Collectively, these observations indicate that finite-time transient instability is accompanied by simultaneous increases in empirical transient-growth
rate and diffusion-induced variability.
%-----------------------------------------------------
\subsection{Persistence and Accumulation of Transient Instability}

Figures~\ref{fig:freq} and~\ref{fig:cumulative}
characterize the temporal persistence of transient
instability during descent through its occurrence
frequency and cumulative accumulation.

Figure~\ref{fig:freq} shows that transient instability
occurs more than twice as frequently in the failed
trajectory as in the successful trajectory, whereas
Fig.~\ref{fig:cumulative} shows that these events
accumulate progressively during terminal descent,
resulting in substantially greater cumulative
instability exposure.

Collectively, these results indicate that mission
outcome is more strongly associated with the
persistence and cumulative accumulation of transient
instability than with isolated transient-growth
events.
%-----------------------------------------------------
\subsection{Quantitative Validation and Engineering Implications}
\label{sec:lunar_metrics}

Table~\ref{tab:lunar_metrics} summarizes the principal
transient-stability metrics extracted from the
flight-like lunar descent telemetry.

\begin{table}[t]
\caption{Transient-stability metrics computed from flight-like lunar descent telemetry. }
\label{tab:lunar_metrics}
\centering
\begin{tabular}{lcc}
\toprule
Metric & Failed & Successful \\
\midrule
Peak Growth Rate $\alpha_{\max}$ & 0.1374 & $7.4\times10^{-5}$ \\
Mean Growth Rate $\bar{\alpha}$ & -0.00349 & -0.00207 \\
Instability Frequency & 0.1073 & 0.0466 \\
Cumulative Instability & 1.5203 & 0.00068 \\
Peak Variability $\beta_{\max}$ & 0.6559 & $<10^{-6}$ \\
\bottomrule
\end{tabular}
\end{table}

The telemetry-derived transient-stability metrics consistently distinguish the successful and failed descent trajectories. Compared with the successful trajectory, the failed trajectory exhibits substantially larger transient-growth rates,
diffusion-induced variability, instability frequency,
and cumulative instability, indicating sustained
finite-time perturbation amplification rather than
isolated transient events.

Collectively, these results are consistent with the
theoretical predictions developed throughout the
paper and demonstrate that the proposed
transient-stability metrics provide information
beyond conventional trajectory-tracking and
long-horizon stability measures. The derived
indicators therefore provide a quantitative basis for
assessing navigation robustness in autonomous
aerospace systems.

Although demonstrated using flight-like lunar descent
telemetry, the proposed methodology is applicable to
a broad class of continuously differentiable
nonlinear stochastic systems, including spacecraft
guidance, navigation, and control, autonomous
robotics, state estimation, autonomous vehicles, and
other safety-critical engineering applications.
%=====================================================
%=====================================================

%=====================================================
%=====================================================
%=====================================================
\section{Conclusion}
\label{sec:conclusion}

This paper developed a finite-amplitude logarithmic
measure framework for finite-time transient stability
analysis of nonlinear It\^o stochastic differential
equations. By extending classical logarithmic
measures from infinitesimal Jacobian dynamics to
finite-amplitude perturbation evolution, the proposed
framework directly characterizes nonlinear transient
growth on stochastic flows without requiring explicit
variational equations or Lyapunov-function
construction.

The resulting theory establishes deterministic and
probabilistic finite-time transient stability through
explicit bounds on the mean and variance of
logarithmic perturbation growth together with
Chernoff-type probabilistic bounds on transient
amplification. The extension to projected stochastic
dynamics further introduces a transient-risk index
that unifies deterministic contraction and
diffusion-induced variability, providing a
computationally tractable framework for
transient-risk-aware analysis and system design.

The proposed framework was validated through
1000 Monte Carlo simulations and flight-like lunar
descent telemetry. The numerical and telemetry
results demonstrate substantially improved prediction
of nonlinear transient growth compared with the
classical Jacobian logarithmic measure and show that
trajectories exhibiting similar nominal behavior may
possess markedly different finite-time transient-risk
characteristics.

Overall, the proposed framework establishes a unified
methodology for finite-time transient stability
analysis, probabilistic transient-risk assessment,
and transient-risk-aware design of nonlinear
stochastic systems. Beyond its application to lunar
descent telemetry, the framework is applicable to a
broad class of autonomous aerospace systems,
including spacecraft guidance, navigation, and
control, planetary landing, rendezvous and docking,
formation flying, and onboard state estimation under
stochastic uncertainty.

Future work will investigate transient-risk-aware
guidance and control, adaptive projection and
data-assimilation strategies, stochastic observer
design, and real-time onboard transient-risk
monitoring for next-generation autonomous aerospace
systems.
%=====================================================
%=====================================================

%%%%%%%%%%%%%%%%%%%%%%%%%%%%%%%%%%%%%%%%%%%%%%%%%%%%%%%%%%%%%%%%%%%%%%%%%%%%%%
%% References
%%%%%%%%%%%%%%%%%%%%%%%%%%%%%%%%%%%%%%%%%%%%%%%%%%%%%%%%%%%%%%%%%%%%%%%%%%%%%%

\bibliographystyle{elsarticle-num}

\bibliography{references}

%%%%%%%%%%%%%%%%%%%%%%%%%%%%%%%%%%%%%%%%%%%%%%%%%%%%%%%%%%%%%%%%%%%%%%%%%%%%%%

\end{document}